\documentclass{amsart}
\usepackage{cleveref}
\usepackage{verbatim}
\usepackage{geometry}
\usepackage{amsfonts}
\usepackage{amssymb}
\usepackage{url}
\usepackage{amsmath}
\usepackage{amsthm}
\newcommand{\THH}{\mathrm{THH}}
\newcommand{\HH}{\mathrm{HH}}
\newcommand{\HP}{\mathrm{HP}}
\newcommand{\TP}{\mathrm{TP}}

\newcommand{\CycSp}{\mathrm{CycSp}}
\newcommand{\HC}{\mathrm{HC}}
\newcommand{\md}{\mathrm{Mod}}
\newcommand{\Sp}{\mathrm{Sp}}
\newcommand{\fun}{\mathrm{Fun}}
\newcommand{\catst}{\mathrm{Cat}_\infty^{\mathrm{perf}}}
\newcommand{\perf}{\mathrm{Perf}}
\usepackage{xy}
\input xy

\renewcommand{\sp}{\mathrm{Sp}}
\xyoption{all}

\renewcommand{\mod}{\mathrm{Mod}}

\theoremstyle{definition}
\newtheorem{definition}{Definition}[section]
\newtheorem{construction}[definition]{Construction}
\newtheorem{example}[definition]{Example}
\newtheorem{proposition}[definition]{Proposition}
\newtheorem{lemma}[definition]{Lemma}
\newtheorem{corollary}[definition]{Corollary}

\newtheorem{remark}[definition]{Remark}
\newtheorem{theorem}[definition]{Theorem}

\begin{document}

\title{Kaledin's degeneration theorem and topological Hochschild homology}
\author{Akhil Mathew}
\date{\today}
\address{Department of Mathematics, University of Chicago, 5734 S. University Ave., Chicago, IL 60637-1514}
\email{amathew@math.uchicago.edu}

\maketitle 

\begin{abstract}
We give a short proof of Kaledin's theorem on  the
degeneration of the 
noncommutative
Hodge-to-de Rham spectral sequence. Our approach is based on topological
Hochschild homology and the theory of cyclotomic spectra. 
As a consequence, we also obtain relative versions of the
degeneration theorem, both in characteristic zero and for regular bases in
characteristic $p$. \end{abstract}

\section{Introduction}

Let $X$ be a smooth and proper variety over a field $k$. 
A basic invariant of $X$ arises from the \emph{algebraic de Rham cohomology},
$\mathrm{H}_{\mathrm{DR}}^*(X)$, given as the hypercohomology of the complex
$\Omega_X^{\ast}$ of sheaves of algebraic differential forms on $X$ with the
de Rham differential. 
Then $\mathrm{H}_{\mathrm{DR}}^*(X)$ is a finite-dimensional graded $k$-vector
space, and is the abutment of the classical \emph{Hodge-to-de Rham} spectral
sequence $\mathrm{H}^i(X, \Omega^j_X)  \implies
\mathrm{H}^{i+j}_{\mathrm{DR}}(X)$ arising from the naive filtration of the
complex of sheaves 
$\Omega_X^{\ast}$. It is a fundamental fact in algebraic geometry that this
spectral sequence degenerates when $k$ has characteristic zero. 
When $k = \mathbb{C}$ and $X$ is K\"ahler, the degeneration arises from Hodge theory. 

After 2-periodization and in characteristic zero, the
above invariants and questions have \emph{noncommutative} analogs, i.e., they
are defined more generally for differential graded (dg) categories rather than only for varieties. 
Let $\mathcal{C}$ be a smooth and proper dg category over a field $k$ 
(e.g., $\mathcal{C}$ could be the derived category $D^b
\mathrm{Coh}(X)$ of a smooth and proper variety $X/k$). 
In this case, a basic invariant of $\mathcal{C}$ is given by the
\emph{Hochschild homology} $\HH(\mathcal{C}/k)$, regarded as a noncommutative version of
differential forms for $\mathcal{C} $ thanks to the classical
Hochschild-Kostant-Rosenberg theorem.

Hochschild homology takes values in the derived category $D(k)$ of $k$-vector spaces;
it produces a perfect complex equipped with an action of the circle $S^1$,
the noncommutative version of the de Rham diferential. 
As a result, one can take the $S^1$-Tate construction to form $\HP(\mathcal{C}/k)
\stackrel{\mathrm{def}}{=} \HH(\mathcal{C}/k)^{tS^1}$, called the \emph{periodic
cyclic homology}
of $\mathcal{C}$ and often regarded as a noncommutative version of de Rham cohomology. 
One has a general spectral sequence, arising from the Postnikov filtration of
$\HH(\mathcal{C}/k)$, 
$\HH_*(\mathcal{C}/k)[u^{\pm 1}] \implies \HP_*(\mathcal{C}/k)$, called the
(noncommutative)
\emph{Hodge-to-de Rham spectral sequence.}
When $\mathcal{C} = D^b \mathrm{Coh}(X)$ for $X$ in characteristic zero, this
reproduces a 2-periodic analog of the Hodge-to-de Rham spectral sequence.

The papers \cite{K1, K2} of Kaledin describe a proof of 
the following result, conjectured by Kontsevich and Soibelman
\cite[Conjecture 9.1.2]{KS}.

\begin{theorem}[Kaledin] 
\label{char0thm}
Let $\mathcal{C}$ be a smooth and proper dg category over a field $k$ of characteristic zero. 
Then the Hodge-to-de Rham spectral sequence $E_2  =\HH_*(\mathcal{C}/k)[u^{\pm 1}]
\implies \HP_*(\mathcal{C}/k)$ degenerates at $E_2$.
\end{theorem} 

An equivalent statement is that the $S^1$-action on $\HH(\mathcal{C}/k)$,
considered as an object of the derived category $D(k)$, is trivial; thus we may
regard the result as a type of \emph{formality} statement. 
Using the comparison between 2-periodic de Rham cohomology and periodic cyclic
homology in characteristic zero, one recovers the classical result that the
(commutative) Hodge-to-de Rham spectral
sequence $\mathrm{H}^i(X, \Omega^j_X) \implies \mathrm{H}^{i+j}_{\mathrm{dR}}(X)$ from Hodge cohomology
to de Rham cohomology degenerates for a smooth and proper variety $X$ in
characteristic zero. 

Kaledin's proof of Theorem~\ref{char0thm} is based on reduction mod $p$. Motivated by the approach of
Deligne-Illusie \cite{DI} in the commutative case, Kaledin proves a formality statement
for Hochschild homology in characteristic $p$ of smooth and proper dg categories 
which satisfy an amplitude bound on Hochschild cohomology and which admit a lifting mod $p^2$. 
Compare \cite[Th. 5.1]{K2} and \cite[Th. 5.5]{K2}. 

In this paper, we will give a short proof of  the following
 slight variant of Kaledin's characteristic $p$ degeneration results. 
Analogous arguments as in \cite{K1, K2} show that this
variant also implies Theorem~\ref{char0thm}.
 \begin{theorem} 
\label{kaledinp2}
\label{mainthm}
Let $\mathcal{C}$ be a smooth and proper dg category over a perfect field $k$ of characteristic $p>
0$. Suppose that: 
\begin{enumerate}
\item  
$\mathcal{C}$ has a lift to a smooth proper dg category over $W_2(k)$. 
\item
$\HH_i(\mathcal{C}/k)$ vanishes for $i \notin [-p,p]$.
\end{enumerate}
 Then the Hodge-to-de Rham
spectral sequence $\HH_*(\mathcal{C}/k)[u^{\pm 1}] \implies \HP_*(\mathcal{C}/k)$
degenerates at $E_2$.
\end{theorem}

We will deduce Theorem~\ref{kaledinp2}
 from the framework of \emph{topological} Hochschild homology and in particular the
theory of cyclotomic spectra as recently reformulated by Nikolaus-Scholze
\cite{nikolaus-scholze}. 
We give an overview of this apparatus in Section 2. 
The idea of using cyclotomic spectra here is, of course, far from new, 
and is already indicated 
in the papers of Kaledin. 

Given $\mathcal{C}$, one considers the topological
Hochschild homology $\THH(\mathcal{C})$ as a module over
the
$\mathbf{E}_\infty$-ring $\THH(k)$, whose homotopy groups are given by
$k[\sigma]$ for $|\sigma| = 2$. One has equivalences of spectra: 
\begin{enumerate}
\item $\THH(\mathcal{C})/\sigma \simeq \HH(\mathcal{C}/k)$. 
\item  
$\THH(\mathcal{C})[1/\sigma]^{(1)} \simeq
\HP(\mathcal{C}/k)$ for smooth and proper $\mathcal{C}/k$.
Here the superscript ${}^{(1)}$ denotes the Frobenius twist. 
\end{enumerate}
The first equivalence is elementary, while the second arises from the cyclotomic
Frobenius and should compare to the ``noncommutative Cartier isomorphisms''
studied by Kaledin. 
These observations imply that the difference between 2-periodic Hochschild homology 
and periodic cyclic homology  (i.e., differentials in the spectral sequence) is
controlled precisely by the presence of $\sigma$-torsion in
$\THH_*(\mathcal{C})$. 
Under the above assumptions of liftability and amplitude bounds, the degeneration
statement then 
follows from an elementary argument directly on the level of $\THH$. 
We formulate this as a general formality statement in
Proposition~\ref{generaldegcrit} below. 

We also apply our methods to prove freeness and degeneration assertions in
Hochschild homology for families of smooth and proper dg categories.
We first review the commutative version. 
If $S$ is a scheme of finite type over a field of characteristic zero and $f\colon X \to
S$ a proper smooth map, then one
knows by a classical theorem of Deligne \cite{Delignedeg} that the relative
Hodge cohomology sheaves $R^i f_* \Omega_{X/S}^j$ form vector
bundles on $S$, and
that the relative Hodge-to-de Rham spectral sequence degenerates when $S$ is
affine.  
When $S$ is smooth, this can be deduced by reduction mod $p$ and a relative
version of the Deligne-Illusie constructions as in \cite{Illusie}. 

There are
noncommutative versions of these relative results, too. 
For example, in characteristic zero, one has the following result. 
\begin{theorem} 
\label{relativedeg:intro}
Let $A$ be a commutative $\mathbb{Q}$-algebra and let $\mathcal{C}$ be a smooth
proper dg category over $A$. Then: 
\begin{enumerate}
\item The Hochschild homology groups $\HH_i(\mathcal{C}/A)$ are finitely
generated projective $A$-modules.  
\item The relative Hodge-to-de Rham spectral sequence degenerates. 
\end{enumerate}
\end{theorem} 

This result can be deduced from Kaledin's theorem. 
When $A$ is smooth at least, the freeness of $\HH_i (\mathcal{C}/A)$  follows
from the existence of a flat connection on periodic cyclic
homology, due to Getzler \cite{getzler}, together with Theorem~\ref{char0thm}. Compare also \cite[Remark 9.1.4]{KS}
for a statement. 
We will give a short proof inspired by this idea, in the form of the nilinvariance
of periodic cyclic homology in characteristic zero and a K\"unneth theorem. 

In fact, we
will formulate the argument as a general formality (and local freeness) criterion
 for $S^1$-actions via cyclotomic spectra. 
This includes the argument for Kaledin's theorem as well as additional
input for the relative case.

\newtheorem*{gc}{Formality criterion}

\begin{gc}
Let $A$ be a commutative $\mathbb{Q}$-algebra and let $M \in \perf(A)^{BS^1}$
be a perfect complex of $A$-modules equipped with an $S^1$-action. Suppose that there
exists a finitely generated $\mathbb{Z}$-algebra $R \subset A$, a dualizable object
$M'$ in the $\infty$-category $\md_{\THH(R)}( \CycSp)$ of $\THH(R)$-modules in
cyclotomic spectra, and an equivalence $M \simeq M' \otimes_{\THH(R)} A$ in
$\perf(A)^{BS^1}$. 
Then the homology groups of $M$ are finitely generated projective $A$-modules
and the $S^1$-action on $M$ is trivial. 
\end{gc}

In characteristic $p$, we can approach relative questions as well using the
cyclotomic Frobenius, although  our methods only apply when the base is smooth. 
 Recent work of Petrov-Vaintrob-Vologodsky
\cite{PVV} has
obtained related statements using the methods of Kaledin and 
the Gauss-Manin connection in periodic cyclic homology. 
In particular, within the range $[-(p-3), (p-3)]$, they obtain a
Fontaine-Laffaille structure on periodic cyclic
homology, which they observe implies projectivity and degeneration.  
\begin{theorem}[{Cf. also \cite[Theorem 1]{PVV}}]
\label{relcharp}
Let $A$ be a regular noetherian $\mathbb{F}_p$-algebra such that the
Frobenius map $A \to A$ is finite. 
Let $\widetilde{A}$ be a flat lift of $A$ to $\mathbb{Z}/p^2$. Let $\mathcal{C}$ be a smooth
and proper dg category over $A$. Suppose that: 
\begin{enumerate}
\item  
$\mathcal{C}$ lifts to a smooth
and proper dg category over $\widetilde{A}$. 
\item
$\HH_i(\mathcal{C}/A)  = 0$  for $i \notin [-(p-1),p-1]$. 
\end{enumerate}
Then the Hochschild homology groups $\HH_i(\mathcal{C}/A)$ are finitely
generated projective $A$-modules and the relative Hodge-to-de Rham spectral
sequence $\HH_*(\mathcal{C}/A)[u^{\pm 1}] \implies \HP_*(\mathcal{C}/A)$
degenerates at $E_2$. 
\end{theorem}

\subsection*{Acknowledgments}

I would like to thank Mohammed Abouzaid, Benjamin Antieau, Bhargav Bhatt, Lars
Hesselholt, Matthew Morrow, Thomas Nikolaus, Alexander Petrov, Nick Rozenblyum,  Peter Scholze, 
and Dmitry Vaintrob
for helpful discussions related to this subject. I would also like to thank
Benjamin Antieau and the referee for several comments on a draft.  This work was done while the
author was a Clay Research Fellow. 

\section{Topological Hochschild homology and cyclotomic spectra}

Let $\mathcal{C}$ be a $k$-linear stable $\infty$-category over a perfect field
$k$
of characteristic $p> 0$. 
A basic invariant of $\mathcal{C}$ 
which we will use essentially in this paper is the 
\emph{topological Hochschild homology} 
$\THH(\mathcal{C})$. 
The construction $\THH(\mathcal{C})$ is
one of a general class of \emph{localizing} invariants of stable $\infty$-categories, including algebraic $K$-theory, and about which there is a significant
literature; compare for example \cite{BGT}.

The construction $\mathcal{C} \mapsto \THH(\mathcal{C})$ is naturally a functor to the homotopy
theory of spectra, and can be refined substantially to the homotopy theory of
\emph{cyclotomic} spectra. 
By definition, $\THH(\mathcal{C} )$ is the Hochschild
homology of $\mathcal{C}$ relative to the sphere spectrum rather than to an
ordinary ring.
As we show below, $\THH(\mathcal{C})$ contains significant information
about the Hochschild homology
$\HH(\mathcal{C}/k)$ and the spectral sequence for $\HP(\mathcal{C}/k)$. 
We begin by giving a brief overview of the relevant structure in this case.

\subsection{Review of stable $\infty$-categories}

We will use the language of stable $\infty$-categories, following 
Lurie \cite[Sec. 1.2]{HA}. 
Furthermore, 
we use the following notation, as in \cite{BGT}. 

\begin{definition}[The $\infty$-category $\catst$] 
We let $\catst$
denote the $\infty$-category of idempotent-complete,
small stable $\infty$-categories, where the morphisms are exact functors. 
\end{definition} 

\newcommand{\einf}{\mathbf{E}_\infty}

Recall 
\cite[Sec. 3.1]{BGT}
that $\catst$ itself acquires the structure of a symmetric monoidal 
$\infty$-category, via the Lurie tensor product \cite[Sec. 4.8]{HA}. Given $\mathcal{C}, \mathcal{D} \in \catst$, the tensor
product $\mathcal{C} \otimes \mathcal{D}$ is universal for the structure that
one has a functor of $\infty$-categories $\mathcal{C} \times \mathcal{D} \to \mathcal{C} \otimes
\mathcal{D}$ which is biexact in each variable. 

Alternatively, one can give an approach to $\catst$ via the (point-set) theory of
\emph{spectrally enriched
categories}, i.e., categories enriched over a good symmetric monoidal
category of spectra, e.g., symmetric spectra \cite{HSS00} or orthogonal spectra
\cite{MM02}. The homotopy theory of
spectrally enriched categories is studied in \cite{Tab09}. See \cite[Theorem
4.23]{BGT} and \cite[Theorem 4.6]{BGTmult} for a comparison with $\catst$.

Next, we recall the theory of $R$-linear $\infty$-categories. 
An example of an object in $\catst$ is the stable $\infty$-category $\perf(R)$
of perfect $R$-modules, for $R$ an $\einf$-ring; in fact, via the $R$-linear
tensor product this is a commutative algebra object in $\catst$.

\begin{definition}[Linear $\infty$-categories] 
Given an $\einf$-ring $R$, one has also a symmetric monoidal $\infty$-category
of \emph{$R$-linear} (idempotent-complete) stable $\infty$-categories (often
abbreviated to ``$R$-linear $\infty$-categories''). 
By definition, this is the $\infty$-category 
$\mod_{\catst}( \perf(R))$
of modules in $\catst$ over the
commutative algebra object $\perf(R) \in \catst$. 
\end{definition} 

See \cite[Appendix D.1]{SAG} for a treatment of the theory. 
We will mostly be interested in the case where $R  = k$ is a field. In this
case, one can also use the more classical theory of differential graded
categories; cf. \cite{cohn} for a comparison.

\subsection{Topological Hochschild homology}
We will use the theory of topological Hochschild homology  for objects in
$\catst$. Treatments (which go through the language of spectral categories) 
appear  in \cite{BM12, BGTmult, ABGHLM}; one can also formulate the construction
purely $\infty$-categorically \cite{ayala-mg-rozenblyum}. 

\begin{construction}[Cyclic bar construction and $\THH$] 
Given a spectrally enriched category $\mathcal{C}$, one defines the 
\emph{topological Hochschild homology} $\mathrm{THH}(\mathcal{C})$ as the geometric realization of the
classical cyclic bar construction on $\mathcal{C}$.

The cyclic bar construction defines a symmetric monoidal functor 
from the category of spectrally enriched categories to the $\infty$-category of
cyclic objects in $\sp$. After taking the geometric realization, it carries Morita equivalences to equivalences of
spectra, and therefore descends to $\catst$. 
It follows that we obtain a symmetric monoidal
functor
\[ \THH: \catst \to \fun(BS^1, \sp).  \]

\end{construction} 

A fundamental feature of $\THH$ (which, as we discuss below, is not shared by 
ordinary Hochschild homology)
is that it acquires a lift to the
$\infty$-category of \emph{cyclotomic spectra}, studied by many authors
including
\cite{BMcyc, nikolaus-scholze, ayala-mg-rozenblyum}. 
We follow the elegant definition of \cite{nikolaus-scholze}, which agrees with
those of the
others authors in the bounded-below case. 

\begin{definition}[Cyclotomic spectra] 
A ($p$-typical) \emph{cyclotomic spectrum} consists 
of an object $X \in \fun(BS^1, \sp)$ together with a map $\varphi: X \to
X^{tC_p}$ in $\fun(BS^1, \sp)$, where we 
regard $X^{tC_p}$ as a spectrum with an $S^1 \simeq S^1/C_p$-action. 
We let $\CycSp$ denote the presentably symmetric monoidal stable
$\infty$-category of cyclotomic spectra. 
\end{definition} 

The topological Hochschild homology of $\mathcal{C} \in \catst$ can be refined
to a cyclotomic spectrum. 
In fact, topological Hochschild homology yields a  symmetric monoidal functor 
$$\THH: \catst \to \CycSp.$$

\newcommand{\symsp}{\mathrm{Spectra}^{\mathrm{O}}}
\begin{construction}[$\THH$ of spectral categories and stable
$\infty$-categories] 
We briefly sketch a construction 
of topological Hochschild homology of an object of $\catst$ as a cyclotomic
spectrum, following \cite{nikolaus-scholze}; it is also
possible to give a  $\infty$-categorical construction as in
\cite{ayala-mg-rozenblyum}.  
Since $\catst$ is obtained as a localization of spectrally enriched
categories, it suffices to carry this construction out for a spectrally enriched
category, where one has a well-defined set of objects. 

Let $\mathfrak{C}$ be a spectrally enriched category, i.e., a category enriched
over the category of orthogonal spectra, $\symsp$. 
In this case, one constructs the Hochschild-Mitchell cyclic nerve
$$N^{\mathrm{cyc}}(\mathfrak{C}): \Lambda \to \symsp,$$ 
for $\Lambda$ the cyclic category. 
Let $\Lambda_p \to \Lambda$ be the edgewise subdivision  (e.g., \cite[Appendix
B]{nikolaus-scholze}). 
Its $p$th edgewise subdivision $\mathrm{sd}_p ( N^{\mathrm{cyc}}(\mathcal{C}))$
yields a functor $\Lambda_p \to \symsp$. 
Unwinding the definitions and using the Tate diagonal as in
\cite{nikolaus-scholze}, one obtains 
a map in the $\infty$-category of cyclic spectra,
$N^{\mathrm{cyc}}(\mathfrak{C}) \to  ( N^{\mathrm{cyc}}( \mathfrak{C}) \circ
\mathrm{sd}_p)^{tC_p} $. 
Taking geometric realizations, one obtains the 
cyclotomic structure $\varphi: \THH(\mathfrak{C}) \to \THH(\mathfrak{C})^{tC_p}$, and all constructions are lax
symmetric monoidal. 
\end{construction}

We now specialize to the $R$-linear case, where $R$ is an $\einf$-ring. 
Since $\THH$ is a symmetric monoidal functor, 
it follows that 
if $\mathcal{C}$ is an $R$-linear stable $\infty$-category, then 
$\THH(\mathcal{C})$ is a module in $\CycSp$ over $\THH(R) = \THH(\perf(R)),$
which is an $\einf$-algebra in $\CycSp$.  
Moreover, $\THH$ defines a symmetric monoidal functor from $R$-linear stable
$\infty$-categories to $\mod_{\CycSp}( \THH(R))$.

\begin{construction}[Relative Hochschild homology] 
Let $R$ be an $\einf$-ring and let 
$\mathcal{C}$ be an $R$-linear $\infty$-category.
We define the \emph{relative
Hochschild homology} $\HH(\mathcal{C}/R) \in \fun(BS^1, \mod(R))$ as the
relative tensor product
\begin{equation} \label{THHtoHH} \HH(\mathcal{C}/R) =\THH(\mathcal{C})
\otimes_{\THH(R)} R , \end{equation}
where we use the canonical $S^1$-equivariant map $\THH(R) \to R$. 
\end{construction} 

\begin{remark} 
Suppose $R$ arises from a commutative orthogonal ring spectrum $R^o$. 
For an $R$-linear $\infty$-category $\mathcal{C}$ presented via 
an $R^o$-spectral category $\mathfrak{C}$ (i.e., a category enriched over
$R^o$-modules), it follows by comparing cyclic bar constructions  and using the
symmetric monoidality of geometric realizations that the above agrees with the
usual definition of $\HH(\mathcal{C}/R)$. For example, 
when $R = k$, this agrees with the usual definition of Hochschild homology for a
dg category. 
\end{remark} 
\begin{remark}[Properties of $\THH$] 
The primary focus of the paper is  about the Hochschild homology of smooth and
proper $k$-linear categories, for $k$ a field. Nonetheless, the use of topological Hochschild
homology appears for the following two reasons: 
\begin{enumerate}
\item $\THH$ is a more primitive invariant: according to  \eqref{THHtoHH}, we can recover
Hochschild homology from $\THH$. 
\item $\THH$ has the additional structure given by the cyclotomic Frobenius
$\varphi$, which does not exist on ordinary Hochschild homology.  
\end{enumerate}
\end{remark}

\subsection{Topological Hochschild homology over $k$}

We will especially be interested in the topological Hochschild homology of a
$k$-linear $\infty$-category, for $k$ a perfect field of characteristic $p> 0$,
which exhibits some special features. 
A basic input here is the calculation in the case when $\mathcal{C}
= \perf(k)$, recalled below (cf. \cite[Sec. 5]{HM97}). \begin{theorem}[B\"okstedt]
\label{bok}
 \( \THH_*(k) \simeq k[\sigma], \quad |\sigma| = 2.  \)
\end{theorem}

\begin{remark} 

Theorem~\ref{bok} shows that $\THH$ can be controlled in a convenient manner. 
A more naive variant of the construction $\mathcal{C} \mapsto \THH(\mathcal{C})$ is to consider the Hochschild
homology $\HH(\mathcal{C}/\mathbb{Z})$ over the integers. Since (by a straightforward
calculation) $ \HH_*(\mathbb{F}_p/\mathbb{Z}) \simeq \Gamma(\sigma)$ is a divided power algebra on a
degree two class, the 
construction of $\THH$ should be regarded as an ``improved'' version of
Hochschild homology over $\mathbb{Z}$. 
\end{remark}

As in \eqref{THHtoHH}, one 
has 
the relation
\begin{equation} \label{basechange} \THH(\mathcal{C}) \otimes_{\THH(k)} k \simeq \HH(\mathcal{C}/k).
\end{equation}
As a result of \eqref{basechange}, $\THH(\mathcal{C})$ can be thought of as a 
one-parameter deformation of $\HH(\mathcal{C}/k)$ over the element $\sigma$.

Recall $\THH(\mathcal{C})$ inherits an action of the circle $S^1$. The
circle  also  acts on $\THH(k)$ (considered as an $\mathbf{E}_\infty$-ring
spectrum), and $\THH$ provides a symmetric monoidal functor
\[ \left\{k\text{-linear \ stable} \ \infty \text{-categories}\right\}  \to
\mathrm{Mod}_{\THH(k)}( \Sp^{BS^1}),  \]
i.e., into the $\infty$-category of spectra with $S^1$-action equipped with a compatible
$\THH(k)$-action.
Using this, one can define the following (which can be thought of as a
noncommutative version of crystalline cohomology).

\begin{definition}[Hesselholt \cite{hesselholt-tp}]
The \emph{periodic topological cyclic homology}
of $\mathcal{C}$ is given by $\TP(\mathcal{C}) = \THH(\mathcal{C})^{tS^1}$. 
\end{definition}

A result of \cite{BMS2} (see also \cite[Sec. 3]{AMN}) shows that
$\TP$ provides a lift to characteristic zero
of 
the periodic cyclic homology $\HP(\mathcal{C}/k)$. 
For example, $\TP_*(k)  \simeq W(k)[x^{\pm 1}]$ for $|x|  = -2$, and in general
one has a natural equivalence of $\TP(k)$-modules 
\begin{equation}\label{TPliftsHP} \TP(\mathcal{C}) \otimes_{\TP(k)} \HP(k) \simeq
\TP(\mathcal{C}) /p \simeq 
\HP(\mathcal{C}/k) .
\end{equation}
The construction $\mathcal{C} \mapsto \TP(\mathcal{C})$ is another extremely
useful invariant one can extract from this machinery. It naturally provides a 
lax symmetric monoidal functor
\[ \left\{k\text{-linear \ stable $\infty$-categories}\right\}  \to
\mathrm{Mod}_{\TP(k)}.  \]
At least for smooth and proper $k$-linear $\infty$-categories, the construction $\TP$ is actually
symmetric monoidal, i.e., satisfies a K\"unneth theorem, by a result of
Blumberg-Mandell \cite{blumberg-mandell-tp} (see also
\cite{AMN}).

In \eqref{TPliftsHP}, we saw that 
periodic cyclic homology can be recovered from $\TP$ by reducing mod $p$. 
Next, we show that we can reconstruct $\HP$ from
$\THH$ in another way. 
Note first that there is a natural map of $\mathbf{E}_\infty$-rings $\TP(k)
\simeq \THH(k)^{tS^1}\to
\THH(k)^{tC_p}$.

\begin{proposition} 
\label{modp}
For $\mathcal{C}$ a $k$-linear stable $\infty$-category, one has an equivalence
of $\TP(k)$-module spectra
$\THH(\mathcal{C})^{tC_p} \simeq \TP(\mathcal{C}) \otimes_{\TP(k)}
\THH(k)^{tC_p} \simeq \HP(\mathcal{C}/k)$. 
\end{proposition} 
For future reference, we actually prove a more general statement. 
\begin{proposition} 

Let $X$ be an arbitrary object of the $\infty$-category
$\mathrm{Mod}_{\THH(k)}( \Sp^{BS^1})$ of modules over $\THH(k)$ in the
symmetric monoidal $\infty$-category of spectra equipped with an
$S^1$-action.\footnote{Compare the discussion in \cite{AMN} for a treatment.}

Then 
the natural map 
of $\TP(k)$-modules
\begin{equation}  X^{tS^1} \otimes_{\TP(k)} \THH(k)^{tC_p}  \to X^{tC_p}.
\label{cc} \end{equation} 
is an equivalence, and one has a natural equivalence of
$\TP(k)$-modules 
 \begin{equation} X^{tS^1}
\otimes_{\TP(k)} \THH(k)^{tC_p} \simeq (X \otimes_{\THH(k)} k)^{tS^1}.
\label{dd}
\end{equation}
\end{proposition} 
\begin{proof} 
To see 
this, we note that there is an $S^1$-equivariant map of $\mathbf{E}_{\infty}$-rings
$\mathbb{Z} \to \THH(\mathbb{F}_p)$, e.g., via the cyclotomic trace (cf. 
\cite[IV.4]{nikolaus-scholze}). One obtains
a 
 square 
of $\mathbf{E}_\infty$-rings
\[ \xymatrix{
\mathbb{Z}^{tS^1} \ar[d]  \ar[r] &  \mathbb{Z}^{tC_p} \ar[d]  \\
\TP(k) \ar[r] &  \THH(k)^{tC_p}
},\]
which one easily checks to be a pushout square.
Now the equivalence \eqref{cc} follows from \cite[Lemma
IV.4.12]{nikolaus-scholze}. To see \eqref{dd}, we use the fact 
that 
$\THH(k)^{tC_p} \simeq \TP(k)/p$ as $\TP(k)$-modules. 
This implies the result 
via the formula $(X \otimes_{\THH(k)} k)^{tS^1} \simeq X^{tS^1}
\otimes_{\TP(k)} k^{tS^1} \simeq X^{tS^1} /p$, which holds because $k =
\THH(k)/\sigma$ belongs
to the thick
subcategory generated by the unit in $\md_{\THH(k)}( \Sp^{BS^1})$
(and which is a generalization of 
\eqref{TPliftsHP}). 
\end{proof}

\subsection{The cyclotomic Frobenius over $k$} 

\begin{example}[{Cf. \cite[IV.4]{nikolaus-scholze} and \cite{HM97}}] 
Suppose $\mathcal{C} = \perf(k)$. In this case, the map
\[ \varphi\colon  \THH(k) \to \THH(k)^{tC_p}  \]
identifies the former with the connective cover of the latter, and
$\pi_* \left(\THH(k)^{tC_p}\right) \simeq k[u^{\pm 1}]$ is a Laurent polynomial ring with $|u| =
2$. The map
$\varphi$ is given by the Frobenius on $\pi_0$ and sends $\sigma \mapsto u$.  
In particular, $\varphi$ induces an equivalence
\[ \THH(k)[1/\sigma] \simeq \THH(k)^{tC_p}.  \]
This computation was originally done by Hesselholt-Madsen \cite{HM97}, 
and we refer to \cite[IV.4]{nikolaus-scholze} for a complete description of
$\THH(k)$ as a cyclotomic spectrum. 
\end{example}

Here $\THH(k) \in \mathrm{CAlg}(\CycSp)$ is a commutative algebra object, and
for $\mathcal{C}$ a $k$-linear stable $\infty$-category, 
$\THH(\mathcal{C})$ is a $\THH(k)$-module. 
The functor $\THH$ yields a symmetric monoidal functor
\[ \left\{k\text{-linear \ stable} \ \infty \text{-categories}\right\}  \to
\mathrm{Mod}_{\THH(k)}( \CycSp).  \]
Note in particular that for a smooth and proper $k$-linear stable
$\infty$-category (cf. \cite[Ch. 11]{SAG} for an account), $\THH$ is therefore a dualizable object of 
$\mathrm{Mod}_{\THH(k)}( \CycSp)$. 
In this paper, all 
our degeneration arguments will take place in the latter $\infty$-category, and
we will often state them in that manner.

We saw above that the 
cyclotomic Frobenius becomes an equivalence on connective covers for $\THH(k)$.
More generally, one can show (cf. \cite{hesselholt} and \cite[Cor. 8.18]{BMS2}) that for a smooth
commutative $k$-algebra, the cyclotomic Frobenius is an equivalence in high enough degrees. 
For our purposes, we need 
a basic observation that in the smooth and proper case, the cyclotomic Frobenius 
becomes an equivalence after inverting $\sigma$. This is a formal dualizability
argument once one knows both sides satisfy a K\"unneth formula. 

\begin{proposition} 
\label{noncommCart}
Let $\mathcal{C}/k$ be a smooth and proper $k$-linear stable $\infty$-category. In this case, the
cyclotomic Frobenius implements an
equivalence
\[ \THH(\mathcal{C})[1/\sigma] \stackrel{\varphi}{\simeq} \THH(\mathcal{C})^{tC_p}
\simeq \HP(\mathcal{C}/k).  \]
More generally, if $X \in \md_{\THH(k)}(\CycSp)$ is a dualizable object, then
the cyclotomic Frobenius implements an equivalence 
\[ X[1/\sigma] \xrightarrow{\varphi} X^{tC_p} \simeq (X \otimes_{\THH(k)}
k)^{tS^1}.  \]
The first equivalence is a $\varphi$-semilinear for the equivalence
$\varphi\colon  \THH(k)[1/\sigma] \simeq
\THH(k)^{tC_p}$, while the second equivalence is $\TP(k)$-linear. 
\end{proposition} 
\begin{proof} 
By Proposition~\ref{modp},  it suffices to prove that $\varphi$ is an isomorphism.
In fact, both the source and target of $\varphi$ are symmetric monoidal functors
from dualizable objects in $\md_{\THH(k)}(\CycSp)$
to the $\infty$-category of $\THH(k)[1/\sigma] \simeq \THH(k)^{tC_p}$-module
spectra (cf. \cite{blumberg-mandell-tp, AMN}) 
and the natural transformation is one of symmetric monoidal functors. Thus the
map is an equivalence for formal reasons \cite[Prop. 4.6]{AMN}. 
\end{proof}

Let $\mathcal{C}$ be smooth and proper over $k$. 
On homotopy groups, it follows that one has isomorphisms of abelian groups $\pi_i
\THH(\mathcal{C})[1/\sigma] \simeq \pi_i \HP(\mathcal{C}/k)$. Both sides are
$k$-vector spaces, and the isomorphism is semilinear for the Frobenius. 
In particular, at the level of $k$-vector spaces, one has a natural isomorphism
\[ \left(\pi_i \THH(\mathcal{C})[1/\sigma] \right)^{(1)} \simeq \HP_i
(\mathcal{C}/k). \]

\begin{remark} 
Suppose $\mathcal{C} = \perf(A)$ for $A$ a smooth commutative $k$-algebra. 
In this case, $\HP(\mathcal{C}/k)$ is related to 2-periodic de Rham cohomology
of
$A$ (see 
\cite{antieauperiodic})
while $\THH(\mathcal{C})[1/\sigma]$ is closely related to  2-periodic
differential forms on $\mathcal{C}$   by  \cite{hesselholt} (and more precisely
by \cite[Cor. 8.18]{BMS2}). 
The relationship between differential forms and de Rham cohomology arising here
is essentially the classical \emph{Cartier isomorphism}, and is made precise in the work of
Bhatt-Morrow-Scholze \cite{BMS2}.

In addition, we expect that 
\Cref{noncommCart} can be compared with the ``noncommutative Cartier
isomorphism'' studied by Kaledin \cite{K1, K2}. 
\end{remark}

\section{The degeneration argument}

In this section, we give the main degeneration argument. 
We begin with the following basic observation and definition. 

Let $R$ be an
$\mathbf{E}_\infty$-ring spectrum over $\mathbb{Z}$ (in this section, $R$ will be a field), and
let $M$ be an $R$-module spectrum equipped with an
$S^1$-action. Suppose the $R$-module $M$ is graded projective. Then the following are
equivalent: 
\begin{enumerate}
\item  
The $S^1$-Tate spectral sequence for $\pi_*( M^{tS^1})$ degenerates. 
\item The $S^1$-action on $M$ (as an $R$-module) is trivial. 
\end{enumerate}

Clearly the second assertion implies the first. To see the converse, we observe
that if the Tate spectral sequence degenerates, then by naturality, the homotopy
fixed point spectral sequence for $\pi_*(M)$ must degenerate too, so that the
map $\pi_*(M^{hS^1}) \to \pi_* (M)$ is surjective. 
Suppose $M$, as an underlying $R$-module, is obtained as the summand $Fe$
associated to an
idempotent endomorphism $e$ of a free $R$-module $F$. If we give $F$ the trivial
$S^1$-action, the degeneration of the homotopy 
fixed point spectral sequence
shows that we can realize the map  $F \to M$ as an $S^1$-equivariant map. 
Restricting now to the summand $Fe$ of $F$, we conclude that $M$ is equivalent
to $Fe$ (with trivial action). 
This is the way in which we regard the  degeneration of the $S^1$-Tate spectral sequence
as a \emph{formality} statement.

\begin{definition} 
Let $k$ be a field. 
Let $M \in \perf(k)^{BS^1}$. We say that $M$ is \emph{formal} if 
the $S^1$-Tate spectral sequence for $M^{tS^1}$ (or equivalently the homotopy
fixed point spectral sequence for $M^{hS^1}$) degenerates at $E_2$. 
This holds if and only if 
\begin{equation}  \label{hheq} \dim_k 
\pi_{\mathrm{even}}( M) = \dim_k \pi_0 M^{tS^1}, 
\quad
\dim_k 
\pi_{\mathrm{odd}}( M) = \dim_k \pi_1 M^{tS^1} \end{equation}
\end{definition}

For the rest of this section, $k$ is a perfect field of characteristic $p> 0$.
We will prove a formality criterion for objects of $\perf(k)^{BS^1}$.
Our main interest, of course, is in the following example; in this section, we will state our
arguments in the more general case of objects in $\perf(k)$ with $S^1$-action,
though. 
Consider a smooth and proper $k$-linear stable $\infty$-category $\mathcal{C}/k$ and its Hochschild
homology $\HH(\mathcal{C}/k)$. One has that $\dim_k \HH_*(\mathcal{C}/k)< \infty$
and that $\HH(\mathcal{C}/k)$ inherits a circle action.

\begin{definition}
We say that the \emph{Hodge-to-de Rham spectral sequence degenerates} for
$\mathcal{C}/k$ if $\HH(\mathcal{C}/k) \in \perf(k)^{BS^1}$ is formal. 
Equivalently, degeneration holds if and only if one has the numerical equalities
$\HH_{\mathrm{even}}(\mathcal{C}/k) = \HP_0(\mathcal{C}/k), \quad 
\dim_k \HH_{\mathrm{odd}}(\mathcal{C}/k) = \HP_1(\mathcal{C}/k)$. 
\end{definition}

One source of objects of $\perf(k)^{BS^1}$ is the $\infty$-category of
dualizable objects of $\md_{\THH(k)}(\CycSp)$. 
Given $X \in \md_{\THH(k)}(\CycSp)$, we have $X \otimes_{\THH(k)} k \in
\md_k^{BS^1}$ and if $X$ is dualizable, then $X \otimes_{\THH(k)} k$ is perfect
as a $k$-module.
For such objects, 
we will translate formality to a statement about $\THH(k)$-modules.
Note that $\HH(\mathcal{C}/k) \in \perf(k)^{BS^1}$ arises in this way, via $X =
\THH(\mathcal{C})$.

First, we need the
following observation about module spectra over $\THH(k)$, which follows from the
classification of finitely generated modules over a principal ideal domain. 
\newcommand{\rank}{\mathrm{rank}}

\begin{proposition} 
\label{perfectmod}
Any perfect $\THH(k)$-module
spectrum is equivalent to a direct sum of copies of suspensions of $\THH(k)$ and 
$\THH(k)/\sigma^n$ for various $n$. 
\end{proposition} 
\begin{proof} 
A perfect $\THH(k)$-module $M$ yields a finitely generated $\THH_*(k)$-module
$\pi_*(M)$. 
Any finitely generated graded $\THH_*(k) = k[\sigma]$-module is a direct sum of
copies of shifts of $k[\sigma]$ and $k[\sigma]/\sigma^n$, for various $n$. It
follows easily that $M$ can be written as a direct sum as desired. 
\end{proof}

The following result now shows that 
degeneration is equivalent to a condition of torsion-freeness on $\THH$.

\begin{proposition} 
\label{torsionfreedeg}
\begin{enumerate}
\item  
Let $X \in \md_{\THH(k)}(\CycSp)$ be dualizable. 
Then $X \otimes_{\THH(k)} k \in \perf(k)^{BS^1}$ is formal if and only if
$X$ is (graded) free (equivalently, $\sigma$-torsion-free) as a $\THH(k)$-module.
\item
If $\mathcal{C}$ is smooth and proper over $k$, 
the Hodge-to-de Rham spectral sequence for $\mathcal{C}$ degenerates
if and only if $\THH(\mathcal{C})$ is free (equivalently,
$\sigma$-torsion-free) as a $\THH(k)$-module.
\end{enumerate}
\end{proposition}

\begin{proof} 
Clearly, the second assertion is a special case of the first. 
It suffices to compare with \eqref{hheq}. In fact, by the equivalence given by
Proposition~\ref{noncommCart}, one sees that 
$\pi_*(X)[1/\sigma]$ is a finitely generated graded free
$\THH(k)_*[1/\sigma]$-module. Moreover, one has
\[ \dim_k \pi_0\left((X \otimes_{\THH(k)} k)^{tS^1})\right) = \dim_k\left( \pi_0 (X[1/\sigma])
\right)= 
\rank_{k[\sigma^{\pm 1}]} \pi_{\mathrm{even}}(X)[1/\sigma],
\]
and similarly for the odd terms.
Thus, formality holds if and only if 
the ranks agree, i.e.,
\[ \rank_{k[\sigma^{\pm 1}]} \pi_{\mathrm{even}}(X)[1/\sigma]
= \dim_k \pi_{\mathrm{even}}(X \otimes_{\THH(k)} k), \quad
\rank_{k[\sigma^{\pm 1}]} \pi_{\mathrm{odd}}(X)[1/\sigma]= \dim_k
\pi_{\mathrm{odd}}(X \otimes_{\THH(k)} k). 
\]
Note that $X \otimes_{\THH(k)} k \simeq X/\sigma$. 
It follows (e.g., using
Proposition~\ref{perfectmod}) that the ranks
(over $\sigma = 0$ and $\sigma$ invertible, respectively) agree if and only if
$X$ is (graded) free as a $\THH(k)$-module spectrum. 
\end{proof}

It thus follows that, in order to verify degeneration of the Hodge-to-de
Rham spectral sequence,  one needs criteria for testing $\sigma$-torsion-freeness in
$\THH_*(\mathcal{C})$. We begin by observing that liftability to the sphere
allows for a direct argument here.
The general idea that liftability to the sphere should simplify the argument was
well-known, and we are grateful to N. Rozenblyum for indicating it to us. 

\begin{example} 
Suppose\footnote{Using the spectral version of the Witt vectors construction,
one can replace $k$ with any perfect field of characteristic $p$. } $k = \mathbb{F}_p$ and 
suppose $\mathcal{C}$ lifts to a stable $\infty$-category
$\widetilde{\mathcal{C}}$
over the sphere $S^0$ (implicitly $p$-completed). 
Note that the map $S^0 \to \THH(\mathbb{F}_p)$ factors through the natural map
$\mathbb{F}_p \to \THH(\mathbb{F}_p)$ given by choosing a basepoint in the
circle $S^1$ via the equivalence $\THH(\mathbb{F}_p) \simeq S^1 \otimes
\mathbb{F}_p$ in $\mathbf{E}_\infty$-rings \cite{MSV}. 
Then, as $\THH(\mathbb{F}_p)$-module spectra, one has an equivalence 
$$\THH(\mathcal{C}) \simeq \THH(\widetilde{\mathcal{C}}) \otimes_{S^0}
\THH(\mathbb{F}_p)
\simeq (\THH(\widetilde{\mathcal{C}}) \otimes_{S^0} \mathbb{F}_p)
\otimes_{\mathbb{F}_p} \THH(\mathbb{F}_p).
$$
Since every $\mathbb{F}_p$-module spectrum is (graded) free, 
this equivalence proves that 
$\THH(\mathcal{C})$ is free as
an $\THH(\mathbb{F}_p)$-module. Thus, degeneration holds for $\mathcal{C}$.
\end{example}

We will now give the argument for a lifting to $W_2(k)$. 
If a $k$-linear stable $\infty$-category $\mathcal{C}$ lifts to $W_2(k)$, then
the $\THH(k)$-module spectrum $\THH(\mathcal{C})$ lifts to $\THH(W_2(k))$. 
By considering the map $\THH(W_2(k)) \to \THH(k)$, we will be able to deduce
$\sigma$-torsion-freeness (and thus degeneration) in many cases. 
The argument will
require a small amount of additional bookkeeping and rely on an amplitude
assumption. 
The basic input is the following fact about the homotopy ring of $\THH(W_2(k))$. 
The entire computation is carried out in \cite{Brun}, at least additively, but we will only
need it in low degrees. 
For the reader's convenience, we include a proof. 
\begin{proposition}[Compare \cite{Brun}] 
Let $k$ be a perfect field of characteristic $p$. 
\label{THHw2}
\begin{enumerate}
\item  We have
\[ \pi_* \tau_{\leq 2p-2} \THH(W_2(k)) \simeq  W_2(k) [ u]/ u^p, \quad
|u| = 2.  \]
\item 
The map $\THH_i(W_2(k)) \to \THH_i(k)$ is zero for $0 < i \leq 2p-2$. 
Furthermore, the map of $\mathbf{E}_\infty$-rings $\THH(W_2(k)) \to \THH(k) \to \tau_{\leq 2p-2} \THH(k)$
factors through the map $k \to \THH(k) \to \tau_{\leq 2p-2} \THH(k)$. 
\end{enumerate}
\end{proposition} 
\begin{proof} 
We compare with Hochschild homology over the integers. 
The map $S^0_{(p)} \to \mathbb{Z}_{(p)}$ induces an equivalence on degrees $<
2p-3$. 
Thus, in the range stated in the theorem, we can compare $\THH$  with Hochschild homology over
$\mathbb{Z}_{(p)}$ or over $W(k)$. 
We have
\[ \HH_*( W_2(k)/\mathbb{Z}_{(p)}) \simeq \Gamma^*_{W_2(k)}[u], \quad |u| = 2,  \]
i.e., the divided power algebra on a class in degree $2$. Indeed, 
we have that 
\[ \HH(W_2(k)/\mathbb{Z}_{(p)}) =
\HH(k/W_2(k)) = W_2(k) \otimes^L_{W_2(k) \otimes^L_{W(k)} W_2 (k)} W_2(k) 
.\]
Since $W_2(k) \otimes^L_{W(k)} W_2(k)$ is the free simplicial commutative ring
over $W_2(k)$ on a class
in degree $1$ (equivalently, a square-zero extension on a class in degree one), 
the Hochschild
homology of $W_2(k)$ is the free simplicial commutative ring over $W_2(k)$ on a class in
degree two, which on homotopy yields a divided power algebra. 

It remains to check that the map $\THH(W_2(k)) \to
\THH( k)$ vanishes on $\pi_2$. 
This, too, follows from the 
comparison with Hochschild homology over $\mathbb{Z}$. 
For a map of commutative rings $A \to B$, let $L_{B/A}$ denote the cotangent complex of $B$
over $A$.
Using the classical Quillen spectral sequence from the cotangent complex to
Hochschild homology (cf., e.g., \cite[Prop. IV.4.1]{nikolaus-scholze}), we
find that the only contributions to 
$\pi_2 \THH(W_2(k)) =\pi_2 \HH(W_2(k)/\mathbb{Z}_{(p)})$
(resp. 
$\pi_2 \THH(k) =\pi_2 \HH(k/\mathbb{Z}_{(p)})$) come from the cotangent complex. 
Thus, one has
to show that the following map vanishes:
\begin{equation} \label{Lmap} \pi_1 L_{W_2(k)/\mathbb{Z}_{(p)}} \to \pi_1
L_{k/\mathbb{Z}_{(p)}}.  \end{equation}
Here one can replace the source $\mathbb{Z}_{(p)}$ with $W(k)$ 
since $k$ is perfect. 
Recall also that if $A$ is a ring and $a \in A$ a regular element, then one
has a natural equivalence $L_{(A/a)
/ A} \simeq (a)/(a^2)[1]$. 
In our setting, one obtains for \eqref{Lmap} the map 
of $W(k)$-modules
\[ (p^2)/(p^4) \to (p)/(p^2),   \]
which is zero. 
Finally, the factorization of the map of $\mathbf{E}_\infty$-rings follows because
$\tau_{\leq 2p-2} \THH(W_2(k))$ is the truncation of the free
$\mathbf{E}_\infty$-ring over $W_2(k)$ on a class in degree two. 
\end{proof} 

We now give an argument that liftability together with a
$\mathrm{Tor}$-amplitude condition implies freeness. The observation is that if
the $\mathrm{Tor}$-amplitude is small, then any torsion has to occur in low
homotopical degree. 
\begin{proposition} 
Let $M$ be a perfect $\THH(k)$-module such that $\pi_i(M) = 0$ for $i < a$.  
Suppose that $M $ lifts to a perfect module over $\THH(W_2(k))$. 
Then multiplication by $\sigma \colon \pi_{i-2}(M) \to \pi_i(M)$ is injective 
for $i \leq a + 2p-2$. 
\label{freecrit}
\end{proposition} 
\begin{proof} 
Without loss of generality, $a = 0$. 
By assumption, 
$M \simeq \widetilde{M} \otimes_{\THH(W_2(k))} \THH(k)$ for some
connective and perfect $\THH(W_2(k))$-module $\widetilde{M}$. 
Truncating, we find that there is a map 
of $\THH(k)$-modules
\begin{equation} \label{eqrange}  M  \to   
\tau_{\leq 2p-2} \widetilde{M} \otimes_{ \tau_{\leq 2p-2} \THH(W_2(k))}
\tau_{\leq 2p-2} \THH(k) 
, 
  \end{equation}
  which induces an isomorphism on degrees $\leq 2p-2$. 
However, by Proposition~\ref{THHw2} and the fact that any $k$-module spectrum
is free,  it follows that the right-hand-side is a free
module over $\tau_{\leq 2p-2} \THH(k)$ on generators in nonnegative degrees. 
This shows that multiplication by $\sigma$ is an injection in this range of
degrees. 
\end{proof}

\begin{proposition} 
\label{dualextend}
Let $M$ be a perfect $\THH(k)$-module with $\mathrm{Tor}$-amplitude
concentrated in $[-p, p]$. Suppose that $M$ lifts to a perfect module over
$\THH(W_2(k))$. Then $M$ is free. 
\end{proposition} 
\begin{proof} 
$M$
is a direct sum of $\THH(k)$-modules
each of which is either free or equivalent to 
$M_{i,j}= \Sigma^{i} \THH(k)/\sigma^j$ for $-p \leq i \leq i + 2j + 1 \leq 
p$ as $M_{i,j}$ has $\mathrm{Tor}$-amplitude in $[i, i+2j+1]$. 
Note that $M_{i,j}$ has an element in $\pi_{i + 2j-2}$ annihilated by $\sigma$,
so we find $i + 2j -2 \geq p-3$ and therefore $i + 2j+1 \geq p$
by Proposition~\ref{freecrit}. 
Therefore, $i + 2j+1 = p$. 

In particular,
we find  that if $M_{i,j}$ occurs as a summand, then $i + 2j + 1 = p$. 
We observe now that if the hypotheses of the lemma apply to $M$, then they
apply to the $\THH(k)$-linear Spanier-Whitehead dual $\mathbb{D}M$: that is,
$\mathbb{D} M$ is a perfect $\THH(k)$-module with $\mathrm{Tor}$-amplitude
concentrated in $[-p, p]$, and such that $\mathbb{D}M $ lifts to a perfect
module over $\THH(W_2(k))$. 
If $M_{i,j}$ occurs as a summand of $M$, then 
its dual, which 
 is given by $\Sigma^{-i - 2j-1} \THH(k)/\sigma^j$, 
occurs as a summand of $\mathbb{D} M$. 
Applying the previous paragraph to $\mathbb{D} M$, it follows also that $-i =
p$. Adding the two equalities, we find that $2j + 1 =
2p$, which is an evident contradiction. 
\end{proof} 

Finally, we can state our general degeneration criterion in characteristic $p$,
which will easily imply 
Theorem~\ref{mainthm}. 

\begin{proposition}[General formality criterion, characteristic $p$] 
\label{generaldegcrit}
Let $k$ be a perfect field of characteristic $p>0$. 
Let $M \in \perf(k)^{BS^1}$ be a perfect $k$-module with $S^1$-action
whose amplitude is contained in $[-p, p]$.
Suppose that there exists a dualizable object $M'  \in \md_{\THH(k)} (\CycSp)$ such
that, as objects in $\perf(k)^{BS^1}$, we have an equivalence $M \simeq M'
\otimes_{\THH(k)} k$. Suppose furthermore that the underlying $\THH(k)$-module
of $M'$ lifts to a perfect module over $\THH(W_2(k))$. Then $M$
is formal. \end{proposition} 
\begin{proof} 
Combine Propositions~\ref{torsionfreedeg} and \ref{dualextend}. 
\end{proof}

\begin{proof}[Proof of Theorem~\ref{mainthm}]
Let $\mathcal{C}$ be a smooth and proper stable $\infty$-category over $k$ satisfying the
assumptions of the theorem.
By assumption, there exists a smooth and proper lift $\widetilde{\mathcal{C}}$ over 
$W_2(k)$ such that $\mathcal{C} \simeq \widetilde{\mathcal{C}}\otimes_{W_2(k)}
k$. Therefore, one has
an equivalence of $\THH(k)$-modules
\[ \THH(\mathcal{C}) \simeq \THH(\widetilde{\mathcal{C}}) \otimes_{\THH(W_2(k))}
\THH(k).  \]
Furthermore, $\THH(\widetilde{\mathcal{C}})$ is a perfect
$\THH(W_2(k))$-module. Now, one can apply Proposition~\ref{generaldegcrit} with $M' =
\THH(\mathcal{C})$. 
\end{proof}

\begin{remark} 

The slight extension of the dimension range via duality goes back to the work of
Deligne-Illusie \cite{DI} and appears in the recent work of Antieau-Vezzosi \cite{AV}
on HKR isomorphisms in characteristic $p$. 
Note also that for a smooth and proper $k$-linear $\infty$-category
$\mathcal{C}$, the Hochschild homology $\HH(\mathcal{C}/k)$ is always
self-dual, cf. 
\cite{SHK}. Hence, it is no loss of generality to assume that the interval in
which the amplitude of Hochschild homology is concentrated is symmetric about
the origin. 
\end{remark} 

We now describe the deduction of Theorem~\ref{char0thm} from
Theorem~\ref{mainthm}, as in \cite{K1, K2}. 
We note that this is a standard argument and is also used in the commutative
case \cite{DI}. We formulate the approach in the following formality
criterion.

\begin{theorem}[General formality criterion, field case] 
\label{char0degcrit}
Let $K$ be a field of characteristic zero. 
Let $M \in \perf(K)^{BS^1}$ be a perfect module equipped with an
$S^1$-action. Suppose that there exists a finitely generated subring $R \subset
K$ and a dualizable object $M' \in \mod_{\THH(R)}( \CycSp)$
such that we have an  equivalence in $\perf(K)^{BS^1}$, 
\( M' \otimes_{\THH(R)} K \simeq M.  \)
Then $M$ is formal. \end{theorem} 
\begin{proof} 
Any finitely generated field extension of $\mathbb{Q}$ is a filtered
colimit of smooth $\mathbb{Z}$-algebras. 
Therefore, $K$ is a filtered colimit of its finitely generated subalgebras
which are smooth over $\mathbb{Z}$.
Enlarging $R$, we can assume that $R$ is smooth over $\mathbb{Z}$. 
Enlarging $R$ further, we can assume that the homology groups of $M' \otimes_{\THH(R)}
R $ (which is a perfect $R$-module spectrum) are finitely generated free $R$-modules and
vanish for $i \notin [-p, p]$, for every prime $p$ which is noninvertible in $R$. 

Suppose that the $S^1$-action on $M$ is nontrivial. Therefore, the $S^1$-action
on $M' \otimes_{\THH(R)} R$  is nontrivial too, and there exists a nontrivial differential in the
Tate spectral sequence for $(M' \otimes_{\THH(R)} R)^{tS^1}$. 
Then we can find a maximal ideal $\mathfrak{m} \subset R$
such that the first differential (which is a map of finitely generated
free
$R$-modules) remains nontrivial after base-change along $R \to
R/\mathfrak{m}$ and thus after base-change along $R \to k
\stackrel{\mathrm{def}}{=}
\overline{R/\mathfrak{m}}$ (i.e., the algebraic closure of the residue field).

Let $M'_k = M' \otimes_{\THH(R)} \THH(k) \in \md_{\THH(k)}(\CycSp)$, which is
a dualizable object. 
Note that  $k$ is a perfect field of
characteristic $p > 0$.
Moreover, the map $R \to k$ lifts to the length
two Witt vectors because $R$ is smooth over $\mathbb{Z}$. 
It follows that the underlying $\THH(k)$-module of $M'_k$
lifts to a perfect $\THH(W_2(k))$-module. 
It follows that $M' \otimes_{ \THH(R)} k \in \perf(k)^{BS^1}$ is formal by  
Proposition~\ref{generaldegcrit}. This contradicts the statement that there is
a nontrivial differential in the Tate spectral 
sequence for $(M' \otimes_{\THH(R)} k)^{tS^1}$ and proves the theorem. 
\end{proof} 

\begin{proof}[Proof of Theorem~\ref{char0thm}] 
Let $\mathcal{C}$  be a smooth and proper stable $\infty$-category over a field $K$ of
characteristic zero. 
By the results of \cite{toen}, there exists a smooth and proper stable
$\infty$-category
$\widetilde{\mathcal{C}}$ over a finitely generated  subalgebra $R
\subset K$  
 such that $\mathcal{C} \simeq \widetilde{\mathcal{C}}
\otimes_R K$. 
Then, one has the dualizable object $\THH(\widetilde{\mathcal{C}}) \in \md_{\THH(R)}(\CycSp)$ and by
base-change, one has  an equivalence in $\perf(K)^{BS^1}$
$\THH(\widetilde{\mathcal{C}}) \otimes_{\THH(R)} K \simeq \HH(\mathcal{C}/K)$. 
Now apply Theorem~\ref{char0degcrit}. 
\end{proof} 

We note that the above arguments actually enable a slight strengthening of 
Theorem~\ref{mainthm}. For example, Theorem~\ref{char0degcrit} easily implies
that if $F \colon \mathcal{C} \to \mathcal{D}$ is a functor of smooth and
proper stable $\infty$-categories over $K$, then the $S^1$-action on the
relative Hochschild homology $\mathrm{fib}( \HH(\mathcal{C}/K) \to
\HH(\mathcal{D}/K))$ is also trivial. More generally, this would work for any
appropriately finite diagram. We formulate this as follows. 

\newcommand{\nmot}{\mathcal{N}\mathrm{Mot}}
Let $K$ be a field of characteristic zero and let $ \nmot_K$ denote the
presentably symmetric monoidal $\infty$-category of noncommutative motives of $K$-linear stable
$\infty$-categories introduced by Tabuada \cite{tab} (see also \cite{BGT, HSS}). 
Since Hochschild homology is an additive invariant, one has a symmetric
monoidal, cocontinuous functor 
\[ \HH(\cdot/K) \colon  \nmot_K \to \md_K^{BS^1},  \]
from $\nmot_K$ into the $\infty$-category $\md_K^{BS^1}$ of $K$-module spectra
equipped with an $S^1$-action. 
Let $\nmot_K^\omega \subset \nmot_K$ denote the thick subcategory generated by the smooth and
proper 
stable $\infty$-categories. 
Recall that if $\mathcal{C}, \mathcal{D}$ are smooth and proper $K$-linear stable
$\infty$-categories, then we have associated objects $[\mathcal{C}],
[\mathcal{D} ] \in \nmot_K^\omega$, and the mapping spectrum is given as 
\[ \mathrm{Hom}_{\nmot_K}( [\mathcal{C}], [\mathcal{D}]) \simeq \mathrm{K}( \fun(\mathcal{C},
\mathcal{D})) , \]
i.e., it is the connective algebraic $\mathrm{K}$-theory spectrum of the $\infty$-category of ($K$-linear)
functors $\mathcal{C} \to \mathcal{D}$. 
Note for instance that given a functor $F\colon \mathcal{C} \to \mathcal{D}$,
one can form the fiber of the associated map $[\mathcal{C}] \to
[\mathcal{D}]$ of noncommutative motives, so that relative Hochschild homology
is given by Hochschild homology of an object of $\nmot_K^\omega$. 
\begin{corollary} 
For any $X \in \nmot_K^\omega$, $\HH(X/K)  \in \perf(K)^{BS^1}$ is formal. 
\end{corollary} 
\begin{proof} 
By the results of \cite{toen}, and the fact that $\mathrm{K}$-theory commutes with
filtered colimits, it follows that 
$\nmot_K^\omega$ is the filtered colimit of the stable $\infty$-categories
$\nmot_R^\omega$ of dualizable noncommutative motives of smooth and proper
$R$-linear $\infty$-categories, as $R$ ranges over the finitely generated
subrings of $K$. 
Thus, there exists $R$ such that
$X$ arises via base-change from a dualizable object $\widetilde{X}$ in the
$\infty$-category $\nmot_R$.  
In this case,  since $\THH$ is an additive invariant of $R$-linear stable
$\infty$-categories into cyclotomic spectra (compare \cite{BMloc, BGT,
ayala-mg-rozenblyum} for treatments), we can similarly form
the dualizable object $\THH(\widetilde{X}) \in \md_{\THH(R)}(\CycSp)$,
which provides a lifting of $\HH(X/K)$. Now we can apply
Theorem~\ref{char0degcrit} as before.  
\end{proof}

\section{Freeness results and degeneration in families}

In this section, we will analyze Hodge-to-de Rham degeneration in families.
In particular, we will give proofs of Theorems~\ref{relativedeg:intro} and
\ref{relcharp}, showing that (under appropriate hypothesis) the relative
Hodge-to-de Rham spectral sequence degenerates and that Hochschild homology
is locally free. In characteristic zero, at least over a smooth base, this result follows from the existence of a connection
\cite{getzler} on periodic cyclic homology together with Theorem~\ref{char0thm}.

Throughout this section, we will need K\"unneth formulas, as in the form
expressed in \cite{AMN}. 
If $(\mathcal{C}, \otimes, \mathbf{1})$ is a symmetric monoidal
stable $\infty$-category with biexact tensor product, 
then an object $X \in \mathcal{C}$ is called \emph{perfect} if it belongs to
the thick subcategory generated by the unit. 
Perfectness is extremely useful to control objects in $\mathcal{C}$ and their
behavior. However, it can be tricky to check directly. 

In \cite{AMN}, the main result is that if $k$ is a perfect field of
characteristic $p > 0$, in the $\infty$-category 
$\md_{\THH(k)}(\Sp^{BS^1})$ of modules over $\THH(k)$ in the $\infty$-category
of spectra with an $S^1$-action, every dualizable object is perfect. 
This in particular implies the K\"unneth theorem for periodic
topological cyclic homology proved by Blumberg-Mandell \cite{blumberg-mandell-tp}. 
In this section, we will 
need variants of this result for non-regular rings in characteristic zero
(Proposition~\ref{dualizableperfect}) and in the perfect (but not necessarily
field) case in characteristic $p$ (Proposition~\ref{perfectperf}). 
This will enable us to control Hochschild homology of stable $\infty$-categories over,
respectively, local Artin rings in characteristic zero and large perfect rings
in characteristic $p$. 

\subsection{Characteristic zero}

In this 
subsection, we explain the deduction of Theorem~\ref{relativedeg:intro}, that the relative
Hodge-to-de Rham spectral sequence degenerates for families of smooth and
proper dg categories in characteristic zero, and that the relative Hochschild homology is locally free. 
We actually prove a result over connective $\mathbf{E}_\infty$-rings and give a
strengthening of the general formality criterion, Theorem~\ref{char0degcrit}. 

The strategy will be to reduce to the local Artinian case, as is standard. We
use the following definition. 
\begin{definition} 
A connective $\mathbf{E}_\infty$-ring $A$ is \emph{local Artinian} if
$\pi_0(A)$ is a local Artinian ring, each homotopy group $\pi_i(A)$ is a
finitely generated $\pi_0(A)$-module, and that $\pi_i(A) =0 $ for $i \gg 0$. 
\end{definition} 

Fix a field $k$ of characteristic zero. 
Let $A$ be a local Artin $\mathbf{E}_\infty$-ring with residue field $k$. 
Note that $A \to k$ admits a section unique up to homotopy by formal smoothness, compare, e.g.,
\cite[Prop. 2.14]{residuefields}, and so we will consider $A$ as an
$\mathbf{E}_\infty$-algebra over $k$. 
Our first goal is to prove K\"unneth formulas for negative and periodic cyclic
homology for smooth and proper stable $\infty$-categories over $A$. 

Following \cite{AMN}, we translate this into  the following statement. 
As in section 2, $\HH(A/k)$ defines a commutative algebra
object in 
the $\infty$-category $\Sp^{BS^1}$ of spectra with an $S^1$-action\footnote{One
could work in the derived $\infty$-category $D(k)$ in this subsection.}
and we can consider the symmetric monoidal $\infty$-category of
modules $\md_{\HH(A/k)} (\Sp^{BS^1})$. 
Given an $A$-linear stable $\infty$-category $\mathcal{C}$,
the Hochschild homology
$\HH(\mathcal{C}/k)$ defines 
an object in 
$\md_{\HH(A/k)} (\Sp^{BS^1})$. 
The homotopy fixed points $\HH(\mathcal{C}/k)^{hS^1}$ are written
$\HC^-(\mathcal{C}/k)$ and called the \emph{negative cyclic homology} of
$\mathcal{C}$ (over $k$). See also \cite{hoyoisfixed} for comparisons with more
classical definitions. 

\begin{proposition} 
\label{dualizableperfect}
Any dualizable object in the symmetric monoidal $\infty$-category
$\md_{\HH(A/k)}( \Sp^{BS^1})$ is perfect. 
\end{proposition} 
\begin{proof} 
Let $M \in \md_{\HH(A/k)}(\Sp^{BS^1})$ be a dualizable object. 
We have a lax symmetric monoidal functor
\[ F\colon  \md_{\HH(A/k)} ( \Sp^{BS^1}) \to \md_{\HC^-(A/k)}, \quad N \mapsto N^{hS^1}.  \]
By
\cite[Sec. 7]{MNN17}, $\md_k(\Sp^{BS^1})$ is identified (via the functor $X
\mapsto X^{hS^1}$) with the
$\infty$-category of $C^*(BS^1; k)$-modules complete with respect to the
augmentation $C^*(BS^1; k) \to k$. 
Taking modules over $\HH(A/k)$, 
we conclude that the functor $F$ is fully faithful. 
Equivalently, the left adjoint functor
\[ \md_{\HC^-(A/k)}  \to  \md_{\HH(A/k)} ( \Sp^{BS^1})\]
is a symmetric monoidal  localization.

To check the claim, it suffices to prove that the functor $F$ is strong symmetric monoidal on
dualizable objects by \cite[Lemma 7.18]{MNN17}. 
That is, for dualizable objects $M, N \in \md_{\HH(A/k)}( \Sp^{BS^1})$, one needs the
map
\begin{equation} \label{kunnethmap} F(M) \otimes_{\HC^-(A/k)} F(N) \to F( M
\otimes  N)  \end{equation}
to be an equivalence of $\HC^-(A/k)$-module spectra. 
Note that we have an element $x \in \pi_{-2} \HC^-(A/k)$
(i.e., a generator of $\pi_{-2} \HC^-(k/k) \simeq \pi_{-2} C^*(BS^1; k)$) such that $\HC^-(A/k)/x
\simeq \HH(A/k)$ and one has an
equivalence of $\HH(A/k)$-module spectra $F(M)/x \simeq
M$ for any $M \in \md_{\HH(A/k)}( \Sp^{BS^1})$
(cf. 
\cite[Sec. 7]{MNN17}). 
It thus follows that \eqref{kunnethmap} becomes an equivalence after
base-change $\HC^-(A/k) \to \HH(A/k)$. 

It thus suffices to show that \eqref{kunnethmap} becomes an equivalence after
inverting $x$. Now we have
\[ (F(M) \otimes_{\HC^-(A/k)} F(N))[1/x] \simeq M^{tS^1} \otimes_{\HP(A/k)}
N^{tS^1}, \quad F(M \otimes N)[1/x] \simeq (M \otimes_{\HH(A/k)}  N)^{tS^1}. \]
In other words, it suffices to show that the functor 
\[ F'\colon  \md_{\HH(A/k)} ( \Sp^{BS^1}) \to \md_{\HP(A/k)}, \quad N \mapsto N^{tS^1}.  \]
is strong symmetric monoidal on dualizable objects. 

However, by Lemma~\ref{nilinv} below, it follows that $F'$ can be identified with the
functor $M \mapsto (M \otimes_{\HH(A/k)} k)^{tS^1}$, i.e., $F'$ factors
through the symmetric monoidal functor 
$\md_{\HH(A/k)} ( \Sp^{BS^1}) \to  \md_k(\Sp^{BS^1})$ given by base-change
$\HH(A/k) \to k$. 
Furthermore,
$\HP(A/k) \simeq k^{tS^1}$. Since dualizable
objects in $\md_k(\Sp^{BS^1})$ are perfect, it follows that $ F'$ satisfies a K\"unneth formula. 
This implies the result. 
\end{proof}

\begin{lemma} 
\label{nilinv}
If $M$ is an object of $\mathrm{Mod}_{\HH(A/k)}(
\Sp^{BS^1})$ 
such that $M$ is bounded below, then the natural map $M \to M
\otimes_{\HH(A/k)} k$
induces an equivalence on $S^1$-Tate constructions.
\end{lemma} 
\begin{proof}

Now $M \simeq \varprojlim \tau_{\leq n} M$ and $M \otimes_{\HH(A/k)} k \simeq
\varprojlim (\tau_{\leq n}M \otimes_{\HH(A/k)} k)$. Both of these inverse limits
become constant in any given range of dimensions. Therefore, they commute with
$S^1$-Tate constructions. 
Therefore, it suffices to assume that $M$ is $n$-truncated. By a further d\'evissage, we can assume that $M$ is
actually a discrete $k$-module, considered as a $\HH(A/k)$-module via the
augmentation. 
We are thus reduced to showing that if $N$ is a discrete $k$-module, then the map $$N \to N
\otimes_{\HH(A/k)} k \simeq N  \otimes_k (k \otimes_{\HH(A/k)} k ) \simeq N \otimes_k
\HH( k \otimes_A k/k)$$
induces an equivalence on $S^1$-Tate constructions. 

However, since the homology of $k\otimes_A k$ forms a connected graded, commutative Hopf algebra, it follows
that $\pi_*(k \otimes_A k)$ is the tensor product
of polynomial algebras on even-dimensional classes and exterior algebras on
odd-dimensional classes. Therefore, 
$k \otimes_A k$ is a free $\mathbf{E}_\infty$-$k$-algebra $\mathrm{Sym}^* V$
for some $k$-module spectrum $V$ with $\pi_i (V) = 0$ for $i \leq 0$. 
Furthermore, $\HH(k \otimes_A k/k) \simeq \mathrm{Sym}^*( S^1_+
\otimes V)$. The desired equivalence now follows because for $i > 0$, $\mathrm{Sym}^i( S^1_+ \otimes V)$ is a free
module over the group ring $k[S^1]$, and so the terms for $i > 0$  
(as a direct sum of graded free $k[S^1]$-modules, and hence a graded free
$k[S^1]$-module)
do not contribute to the Tate construction. 
\end{proof}

\begin{corollary} 
\label{nilinvsmp}
Let $A$ be a local Artin $\mathbf{E}_\infty$-ring and let $\mathcal{C}$ be a smooth and proper 
stable $\infty$-category over $A$. 
Then the map $\HP(\mathcal{C}/k) \to \HP(\mathcal{C} \otimes_A k/k)$ is an
isomorphism. 
\end{corollary}

Note that when $A = k$ itself, this recovers certain cases of the classical
theorem of Goodwillie (\cite[Theorem~II.5.1]{Good}, \cite[Lemma~I.3.3]{Goodrel}) about the nilinvariance of periodic cyclic
homology. 
The corollary follows from Lemma~\ref{nilinv} because  one has an equivalence
\[ \HH(\mathcal{C} \otimes_A k/k) \simeq \HH(  \mathcal{C}/k )
\otimes_{\HH(A/k)} k.    \]

\begin{corollary} 
\label{localartinfree}
Let $A$ be a local Artin $\mathbf{E}_\infty$-ring. 
Let $M \in \md_{\HH(A/k)}( \Sp^{BS^1})$ be dualizable, and let $M_A = M
\otimes_{\HH(A/k)} A 
\in \md_{A}(\Sp^{BS^1})
$ and $M_k \in M \otimes_{\HH(A/k)}k \in \perf(k)^{BS^1}$. 
Then: 
\begin{enumerate}
\item  
$M_A  \in \md_{A}(\Sp^{BS^1})$ belongs to the thick subcategory
generated by the unit. 
\item $M_A^{tS^1} \otimes_A k \simeq M_k^{tS^1}$. 
\item
$M_A^{tS^1}$ is a graded free $A^{tS^1}$-module. 
\end{enumerate}

\end{corollary} 
\begin{proof} 
By
Proposition~\ref{dualizableperfect}, 
$M$
belongs 
to the thick subcategory generated by the unit in 
$\md_{\HH(A/k)}( \Sp^{BS^1})$. It follows that 
$M_A \in \md_{A}(\Sp^{BS^1})$ belongs to the thick subcategory
generated by the unit. Thus, we  obtain the first claim. 
The second claim is implied by the first, as for any perfect object  
$X \in \md_A( \Sp^{BS^1})$, one has $(X \otimes_A k)^{tS^1} \simeq X^{tS^1}
\otimes_A k$ by a thick subcategory argument. 

Finally,  one has  natural maps
\[  M^{tS^1} \to M_A^{tS^1} \to
M_A^{tS^1} \otimes_A k \simeq M_k^{tS^1} 
,\]
such that the composite is an equivalence 
by Lemma~\ref{nilinv}.
Thus, the map $M_A^{tS^1} \to (M_A^{tS^1}) \otimes_A k$ has a
section of $k$-module spectra. 
Lifting a basis, this implies that $M_A^{tS^1}$ is free as an $A^{tS^1}$-module. 
\end{proof}

\begin{lemma} 
Let $A$ be an augmented local Artin $\mathbf{E}_\infty$-ring with residue field $k$. 
Let $M $ be a perfect $A$-module. Then 
\begin{equation} \label{ineqh} \dim_k ( \pi_*(M) ) \leq 
(\dim_k \pi_*(A)) ( \dim_k \pi_*(M \otimes_A k) )
,\end{equation}
and if equality holds $M$ is free. 
\end{lemma} 
\begin{proof} 
Since $A$ has a filtration 
(in $A$-modules)
by copies of $k$, the inequality is evident. 
If equality holds, suppose that $i \in \mathbb{Z}$ is minimal such that
$\pi_i(M) \neq 0$. 
Then also $\pi_i(M \otimes_A k)  = \pi_i(M) \otimes_{\pi_0(A)} k \neq 0$ by
Nakyama's lemma. 
Choose $x \in \pi_i(M)$ whose image in $\pi_i(M \otimes_A k)
$ is nonzero. Form a cofiber sequence $ \Sigma^i A
\stackrel{x}{\to} M \to N$ of $A$-modules. It follows that 
$$\dim_k( \pi_*(N \otimes_A k)) = \dim_k ( \pi_*(M \otimes_A k)) -1, \quad 
\dim_k ( \pi_*(M)) \leq \dim_k ( \pi_*(N)) + \dim_k  \pi_*(A).$$
Combining this with 
\eqref{ineqh}, we find that $\dim_k \pi_*(N) = (\dim_k \pi_*(A))(\dim_k \pi_*(N
\otimes_A k)) $. By an evident induction, $N$ is free as an $A$-module. The long
exact sequence in homotopy, which must reduce to a short exact sequence, now shows that $M$ is also free as an $A$-module. 
\end{proof}

We can now prove the main freeness and degeneration theorems of this section,
which provides a substantial strengthening of Theorem~\ref{char0degcrit}.
Let $\mathrm{CAlg}( \Sp_{\geq 0})$ denote the $\infty$-category of connective
$\mathbf{E}_\infty$-rings. In the following argument, one could also work with
simplicial commutative rings. 
We will use the notion of a
compact or finitely presented object in the $\infty$-category of connective
$\einf$-rings. 
Recall that this in particular implies that, for any such $R$, $\pi_0(R)$ is
finitely generated as a ring, and $\pi_i(R)$ is a finitely generated
$\pi_0(R)$-module for each $i \geq 0$, cf. \cite[7.2.4]{HA}.  
\begin{theorem}[General formality criterion, relative case] 
\label{strongformality0}
Let $A$ be a connective $\mathbf{E}_\infty$-algebra over $\mathbb{Q}$. 
Let $M \in \perf(A)^{BS^1}$. 
Suppose that there exists a compact object $R \in \mathrm{CAlg}(\Sp_{\geq
0})$ with a map $R \to A$, 
a dualizable object $M'_R
\in \md_{\THH(R)}(\CycSp)$, and an equivalence $M'_R \otimes_{\THH(R)} A \simeq
M \in \perf(A)^{BS^1}$. Then $M$ is a finitely generated graded projective
$A$-module and the $S^1$-action on $M$ is formal. 
\end{theorem} 
\begin{proof} 
We first treat the case where $A$ is a local Artin $\mathbf{E}_\infty$-ring
with residue field $k$. 
To see that $M$ is free, it suffices to show that equality holds in \eqref{ineqh}. 
Our assumptions show that $M$ lifts to a dualizable object of $\md_{\HH(A/k)}
(\Sp^{BS^1})$. 
Using the Tate spectral sequence, one obtains
\begin{equation} \label{relineq}  \dim_k  \pi_0 (M^{tS^1})+ 
\dim_k \pi_1( M^{tS^1}) 
\leq \dim_k \pi_{\mathrm{*}}(M )
.\end{equation}
Moreover, 
by Corollary~\ref{localartinfree}, 
we know that $M^{tS^1}$ is a free $A^{tS^1}$-module and that
$M^{tS^1} \otimes_A k \simeq (M \otimes_A k)^{tS^1}$. 
Note that $\pi_0 R$ is a finitely generated $\mathbb{Z}$-algebra. 
Thus we can apply Theorem~\ref{char0degcrit}, and we find that $M \otimes_A k$ is formal in $\perf(k)^{BS^1}$.
We obtain:
\begin{align*}
\dim_k  \pi_0 (M^{tS^1})+ 
\dim_k \pi_1(M^{tS^1})  & = 
\left( \dim_k  \pi_0 \left( (M \otimes_A k)^{tS^1} \right) + \dim_k \pi_1 
\left( (M \otimes_A k)^{tS^1} \right)
\right)
\dim_k \pi_*(A) \\
& =  \dim_k \pi_* (M \otimes_ A k)  \dim_k \pi_*(A). 
\end{align*}
Combining 
the above two inequalities, we obtain 
$\dim_k \pi_* ( M \otimes_A k)   \dim_k \pi_*(A) \leq \dim_k
\pi_*(M)$, which shows that the converse of
\eqref{ineqh} holds and $M$ is free. 
Moreover, equality holds in \eqref{relineq}, so that the $S^1$-Tate spectral
sequence for $M$ degenerates and $M$ is formal. 

We now treat the general case. 
Clearly it suffices to treat the case where $A$ is a
compact object of the $\infty$-category of connective
$\mathbf{E}_\infty$-algebras over $\mathbb{Q}$. 
In this case, $\pi_0(A)$ is noetherian (as a finitely generated
$\mathbb{Q}$-algebra) and the homotopy groups $\pi_i(A)$ are
finitely generated $\pi_0(A)$-modules. 
We thus suppose $A$ is of this form. 

To check the above statements, it suffices to replace $A$ by its localization at any prime ideal
of $\pi_0(A)$. Thus, we may assume that $\pi_0(A)$ is local. Let $x_1, \dots,
x_n \in \pi_0(A)$ be a system of generators of the maximal ideal. 
For each $r > 0$, we let $A'_r  = A/(x_1^r, \dots, x_n^r)$. 
Note moreover that $A'_r \simeq \varprojlim \tau_{\leq m} A'_r$ and that 
$ \varprojlim_r A'_r $ is the completion of $A$, which is in particular
faithfully flat over $A$. 
By the above analysis, 
$M \otimes_{A} \tau_{\leq m} A'_r$ is a free $\tau_{\leq m}
A'$-module for each $m,r$ and the Tate spectral sequence degenerates. Now we can let $m, r \to \infty$.  Since 
$M$ is perfect as an $A$-module, 
it follows that 
$M$ is free, as desired, and the $S^1$-action is formal.  
\end{proof} 

Let $A$ be a connective $\mathbf{E}_\infty$-algebra over $\mathbb{Q}$.
Similarly, one can construct \cite{HSS} the $\infty$-category $\nmot_A$ of noncommutative
motives of $A$-linear $\infty$-categories. 
We let $\nmot_A^\omega$ denote the thick subcategory generated by the
motives of smooth and proper $A$-linear $\infty$-categories. 
We have, again, a Hochschild homology functor $\HH(\cdot/A) \to \md_A^{BS^1}$. 
The next result gives a basic formality property of this functor; for smooth
and proper $A$-linear $\infty$-categories, it includes the degeneration of the
relative Hodge-to-de Rham spectral sequence. 
\begin{corollary} 
Let 
$X \in \nmot_A^\omega$. Then $\HH(X/A) \in \md_A^{BS^1}$ is a finitely
generated projective $A$-module and the $S^1$-action is formal. 
\end{corollary} 
\begin{proof} 
Here we use a refinement of the results of \cite{toen} for
$\mathbf{E}_\infty$-algebras. 
Namely, we claim that the functor which assigns to an $\mathbf{E}_\infty$-ring spectrum
$R$ the $\infty$-category of smooth and proper $R$-linear $\infty$-categories
commutes with filtered colimits in $R$. 
Now, 
smooth and proper $R$-linear $\infty$-categories are compact; in fact, combine
\cite[Props. 3.5, 3.11]{AG}. Therefore, it suffices to see that if $R$ is a
filtered colimit of $\mathbf{E}_\infty$-algebras $R_i$, then any smooth and
proper $R$-linear $\infty$-category $\mathcal{C}$ descends to some $R_i$. 
To see this, we 
observe that $\mathcal{C}$ is equivalent to $\perf(B)$ for an associative
$A$-algebra $B$ which is compact \cite[Prop. 3.11]{AG}  and we can descend the
algebra  to a compact algebra over some $R_i$ thanks to 
\cite[Lemma 11.5.7.17]{SAG}. Moreover, by compactness we can also descend the
duality datum to some finite stage. 

In view of this, we conclude that given $X \in \nmot_A^\omega$, there exists a
compact object $R \in \mathrm{CAlg}(\Sp_{\geq 0})$ mapping to $A$ and a smooth and proper
$R$-linear $\infty$-category $\widetilde{\mathcal{C}}$ such that $\mathcal{C}
\simeq \widetilde{\mathcal{C}} \otimes_R A$. 
Using Theorem~\ref{strongformality0}, we can now conclude the proof as before. 
\end{proof}

\subsection{Characteristic $p$}
The characteristic zero assertion essentially amounts to the idea 
that periodic cyclic homology should form a crystal over the base which is also coherent, and any
such is necessarily well-known to be locally free. 
In characteristic $p$, one can appeal to an analogous argument: 
given a smooth algebra $R$ in characteristic $p$, any finitely generated
$R$-module $M$ isomorphic to its own Frobenius twist is necessarily locally
free \cite[Prop. 1.2.3]{EK}. In this subsection, we prove 
Theorem~\ref{relcharp} from the introduction. 
In doing so, we essentially use the Frobenius-semilinearity of the
cyclotomic Frobenius.

We first discuss what we mean by liftability. 
Let $A$ be  a regular (noetherian) $\mathbb{F}_p$-algebra. 
Recall that $A$ is \emph{$F$-finite} if the Frobenius map $\varphi\colon  A \to A$ is
a finite morphism. 
We refer to \cite[Sec. 2.2]{DundasMorrow} for a general discussion of $F$-finite rings. 

\begin{definition} 
Given an $F$-finite regular noetherian ring $A$, a \emph{lift} of $A$ to $\mathbb{Z}/p^2$ will mean 
a flat $\mathbb{Z}/p^2$-algebra $\widetilde{A}$ with an isomorphism
$\widetilde{A} \otimes_{\mathbb{Z}/p^2} \mathbb{F}_p \simeq A$. 
\end{definition}

Let $A$ be a regular noetherian $\mathbb{F}_p$-algebra. By Popescu's smoothing
theorem (see \cite[Tag 07GC]{stacks-project} for a general reference), $A$ is a
filtered colimit of smooth $\mathbb{F}_p$-algebras. It follows that the
cotangent complex $L_{A/\mathbb{F}_p}$ is concentrated in degree zero and
identified with the K\"ahler differentials; in addition, they form a flat
$A$-module. If $A$ is in addition $F$-finite,
then the K\"ahler differentials are finitely generated and therefore
projective as an $A$-module. 
Recall that the cotangent complex controls the infinitesimal deformation theory
of $A$ \cite[Ch. III, Sec. 2]{Illusiecotangent}. 
Therefore, $A$ is formally smooth as an $\mathbb{F}_p$-algebra, and a lift to
$\mathbb{Z}/p^2$ exists. 
Given a lift $\widetilde{A}$ to $\mathbb{Z}/p^2$,  it follows that
$\widetilde{A}$ is formally smooth over $\mathbb{Z}/p^2$. In particular, it
follows that any two lifts to $\mathbb{Z}/p^2$ are (noncanonically) isomorphic. 
Moreover, if $A \to B$ is a map of $F$-finite regular noetherian
$\mathbb{F}_p$-algebras and $\widetilde{A}, \widetilde{B}$ are respective lifts to
$\mathbb{Z}/p^2$, then the map lifts to a map $\widetilde{A} \to \widetilde{B}$.

Let $A$ be a regular $F$-finite $\mathbb{F}_p$-algebra. Then the Frobenius $\varphi\colon  A \to A$ is a
finite, flat morphism. We let $A_{\mathrm{perf}}$ denote the \emph{perfection}
of $A$, i.e., the colimit of copies of $A$ along the Frobenius map. 
Then we have inclusions
\[ A \subset A^{1/p} \subset A^{1/p^2} \subset \dots A_\mathrm{perf},  \]
such that all maps are 
faithfully flat and the colimit is $A_\mathrm{perf}$. Our strategy will
essentially be descent to $A_\mathrm{perf}$. Unfortunately, 
$A_\mathrm{perf}$ is not noetherian. Thus, we will need the following
result. 

\begin{proposition} 
Let $A$ be a regular $F$-finite $\mathbb{F}_p$-algebra. 
\label{coherentlemma}
\begin{enumerate}
\item  Then the ring
$A_{\mathrm{perf}}$ is coherent, i.e., the finitely presented modules form
an abelian category.
\item
Let $I \subset A$ be an ideal.  Given a 
finitely presented $A_\mathrm{perf}$-module $M$, the submodule $M' \subset M$
consisting of those elements annihilated by a power of $I$ is also coherent and
its annihilator in $A_\mathrm{perf}$ is finitely generated. 
\end{enumerate} 
\end{proposition} 
\begin{proof} 
The first assertion follows because $A_\mathrm{perf}$ is the filtered colimit
of copies of the noetherian ring $A$ along 
the Frobenius map, which is flat by regularity. 
If $M$ is a coherent $A_{\mathrm{perf}}$-module, then $M$ 
descends to $A^{1/p^n}$ for some $n$, i.e., there exists a finitely generated
module $M_n$ over $A^{1/p^n}$ such that $M \simeq A_{\mathrm{perf}}
\otimes_{A^{1/p^n}}
M_n$. Then $M_n$ has an $A^{1/p^n}$ submodule $M'_n$ consisting of the
$I$-power torsion, which is also finitely generated 
(and hence finitely presented), and such that the quotient has no $I$-power
torsion. 
It follows from flatness that $M'_n \otimes_{A^{1/p^n}}
A_{\mathrm{perf}} = M'$, which is thus coherent. 
Since $M'$ is coherent, its annihilator ideal is also coherent. 
\end{proof} 

We will also need to observe that analogs of B\"okstedt's calculation of
$\THH(k)$ hold when $k$ is any perfect $\mathbb{F}_p$-algebra, not only a field. 
Similarly, analogs of Propositions~\ref{THHw2} and Proposition~\ref{freecrit}
hold with analogous arguments. 
\begin{proposition} \label{freenotperf}
Let $k$ be a perfect $\mathbb{F}_p$-algebra. Suppose $M$ is a perfect
$\THH(k)$-module and $\pi_i(M/\sigma) = 0$ for  $i \notin [a,b]$ for $b-a \leq 2p-2$. 
Suppose $M$ lifts to a perfect $\THH(W_2(k))$-module. Then, as $\pi_* \THH(k)
\simeq k[\sigma]$-modules, one has $\pi_*(M) \simeq \pi_*(M/\sigma) \otimes_k
k[\sigma]$. 
\end{proposition}
\begin{proof} 
Without loss of generality, $a = 0$. 
For each $j$, we need to argue that multiplication by $\sigma$ is a \emph{split} injection of
$k$-modules $\pi_{j-2}(M) \to \pi_j(M)$. 
Using the long exact sequence and the assumption on $M/\sigma$, we find that
multiplication by $\sigma$ is an isomorphism for $j \geq 2p-1$. 
For $j \leq 2p-2$, the equivalence in the range $[0, 2p-2]$ as in \eqref{eqrange} implies the result. 
\end{proof} 

We can now state and prove the main formality statement in characteristic $p$
over a regular $F$-finite base. 
\begin{theorem}[Formality criterion, relative characteristic $p$ case] 
Let $A$ be a regular $F$-finite $\mathbb{F}_p$-algebra. 
Let $\widetilde{A}$ be a flat lift to $\mathbb{Z}/p^2$. 
Let $M \in \perf(A)^{BS^1}$. Suppose
that: 
\begin{enumerate}
\item  
There exists a dualizable object $M' \in \md_{\THH(A)}( \CycSp)$ such that $M'
\otimes_{\THH(A)} A \in \perf(A)^{BS^1}$. 
\item 
$\pi_i(M)$ vanishes for $i \notin [a, b]$ for some $a,b$ with $b-a \leq 2p-2$. 
\item The underlying $\THH(A)$-module of $M'$ lifts to a
perfect $\THH(\widetilde{A})$-module.
\end{enumerate}
Then $M$ is a finitely generated (graded) projective $A$-module, and the
$S^1$-action on $M$ is formal. 
\end{theorem}

\begin{proof} 
First, we  can reduce to the case where  $A$ is an $F$-finite regular local ring with maximal ideal $\mathfrak{m}$.
In this case, we can induct on the Krull dimension $d$ of $A$. We can assume that
the result holds for all $F$-finite regular local 
rings of Krull dimension less than $d$. 
When $d  = 0$, the claim is of course 
Theorem~\ref{kaledinp2}. 

 To verify the claims for $A$, we can now replace $A$ by its
 $\mathfrak{m}$-adic completion $\widehat{A}$, which is faithfully flat over
 $A$. Note that $\widehat{A}$ is also an $F$-finite regular local ring of Krull
 dimension $d$. 
 Since $\widehat{A}$ is complete, it contains a copy of its residue field $k$
 and is identified with $\widehat{A} \simeq k[[x_1, \dots, x_n]]$. 
We can consider the faithfully flat map $\widehat{A} \to
k_{\mathrm{perf}}[[x_1, \dots, x_{d}]]$. 
Replacing $A$ with 
$k_{\mathrm{perf}}[[x_1, \dots, x_{d}]]$, we will now 
simply assume that $A$ is in addition complete and has perfect residue field. 
By the inductive hypothesis, all the differentials in the Hodge-to-de Rham
spectral sequence are $\mathfrak{m}$-power
torsion and that $\HH(\mathcal{C}/A)$ is locally free away from $\mathfrak{m}$.

Let $A_{\mathrm{perf}}$ denote the (colimit) perfection of $A$, so one has a
faithfully flat map
$A \to A_{\mathrm{perf}}$. 
We form the base-changes
$M'_{\mathrm{perf}} \stackrel{\mathrm{def}}{=} M' \otimes_{\THH(A)}
\THH(A_{\mathrm{perf}}) \in \md_{\THH(A_{\mathrm{perf}})}(\CycSp)$ (which is
a dualizable object) and $M_{\mathrm{perf}} = M \otimes_A
A_{\mathrm{perf}} \in \perf(A_\mathrm{perf})^{BS^1}$. 
We claim that the cyclotomic Frobenius
\[ \varphi\colon  M'_{\mathrm{perf}}[1/\sigma] \to  
(M'_{\mathrm{perf}})^{tC_p} \simeq (M_{\mathrm{perf}})^{tS^1}
\]
is an equivalence. This follows using the same arguments as in \cite[Sec.
4]{AMN}; 
again, one needs to know that both sides are symmetric monoidal functors in
$M'_{\mathrm{perf}}$. 
For this, it suffices to show that $M'_{\mathrm{perf}}$
belongs to the thick subcategory generated by the 
unit in $\mathrm{Mod}_{\THH(A_{\mathrm{perf}})}(\Sp^{BS^1})$. 
We will check this in Proposition~\ref{perfectperf} below. 

Note that  $M'_{\mathrm{perf}}$ is an
$\THH(A_{\mathrm{perf}})$-module, and 
$M'_{\mathrm{perf}}/\sigma \simeq M_{\mathrm{perf}}$.
Under the liftability hypotheses, we conclude  using
Proposition~\ref{freenotperf}
that 
there is an isomorphism of $A_{\mathrm{perf}}[\sigma]$-modules
\[ \pi_*(M'_{\mathrm{perf}}) \simeq \pi_*(M_{\mathrm{perf}})[\sigma]. \]
Combining, we find an isomorphism of $A_{\mathrm{perf}}$-modules
\begin{equation} \label{isomorphism} 
\pi_* \left( M_{\mathrm{perf}} \right)[\sigma^{\pm 1}]^{(1)} \simeq \pi_* \left(
M_{\mathrm{perf}}^{tS^1}
\right).
\end{equation}

In addition, we have the Tate spectral sequence, which shows that 
$\pi_0 (M_{\mathrm{perf}}^{tS^1})$ is a subquotient of 
$\pi_{\mathrm{even}}(M_{\mathrm{perf}})$ and
is a coherent $A_\mathrm{perf}$-module. Since
the differentials are $\mathfrak{m}$-power torsion, it follows that 
the $\mathfrak{m}$-power torsion in 
$\pi_0 (M_{\mathrm{perf}}^{tS^1})$
 is a subquotient of
the $\mathfrak{m}$-power torsion
in 
$\pi_{\mathrm{even}}(M_{\mathrm{perf}})$.

Let $I$ be the annihilator of the $\mathfrak{m}$-power torsion in 
$\pi_{\mathrm{even}}(M_{\mathrm{perf}})$,
which by Proposition~\ref{coherentlemma} is a finitely generated ideal. 
Then combining the above observations and \eqref{isomorphism}, 
we find that $I^{[p]}$ (i.e., the ideal generated by $p$th powers of elements
in $I$) is the annihilator of the $\mathfrak{m}$-power torsion in 
$\pi_0 (M_{\mathrm{perf}}^{tS^1})$. 
Since this is a subquotient of 
$\pi_{\mathrm{even}}(M_{\mathrm{perf}})$, it
follows that $I\subset I^{[p]}$, which is only possible for a finitely
generated proper ideal if $I = (0)$. 
Therefore, 
$\pi_{\mathrm{even}}(M_{\mathrm{perf}})$
(and
similarly for the odd-dimensional Hochschild homology) is torsion-free. 

Finally, it suffices to prove freeness. 
We have proved that $\pi_*(M)$ consists of finitely generated,
torsion-free $A$-modules. 
Let $x \in \mathfrak{m} \setminus \mathfrak{m}^2$, so that $A/x$ is a regular
local ring too. 
It follows that $\pi_*(M/x)$ is $x$-torsion-free and that, by
induction on the Krull dimension,  
$\pi_*(M)/(x)$ is a free $A/(x)$-module. 
This easily implies that 
$\pi_*(M)$ is free as an $A$-module. 
By comparing with the base-change from $A$ to the perfection of its fraction
field, it also follows that $M$ is formal. 

\end{proof} 

In the course of the above argument, we had to check  a perfectness statement. 
In \cite{AMN}, such results are proved when $A_\mathrm{perf}$ is a field, but
they depend on noetherianness hypotheses. One can carefully remove the
noetherianness hypotheses in this case, but for simplicity, we verify this by
using the technique of relative $\THH$ (also discussed in \cite[Sec. 3]{AMN}). 
The starting point is a relative version of B\"okstedt's calculation. 
We denote by $S^0[q_1, \dots, q_n]$ the $\mathbf{E}_\infty$-ring
$\Sigma^\infty_+ ( \mathbb{Z}_{\geq 0}^n)$. 
The idea of considering $\THH$ relative to such $\mathbf{E}_\infty$-rings is
known to experts, and plays an important role in \cite{BMS2}.
\begin{proposition} 
\label{relTHH}
Let $A$ be an $F$-finite regular local ring with system of parameters $t_1,
\dots, t_n$ and perfect residue field $k$. Consider the map 
of $\mathbf{E}_\infty$-rings
\( S^0[q_1, \dots, q_n] \to A , \quad q_i \mapsto t_i.  \)
Then $$\THH(A/S^0[q_1, \dots, q_n])_* \simeq A[\sigma], \quad |\sigma | = 2,$$
where $\sigma$ is the image of the B\"okstedt element under the natural map 
$\THH(\mathbb{F}_p) \to \THH(A/S^0[q_1, \dots, q_n])$. 
\end{proposition} 
\begin{proof} 
Compare also the treatment in \cite[Sec. 3]{AMN}. 
Since $A$ is $F$-finite and regular, the cotangent complex
$L_{A/\mathbb{F}_p}$ is a finitely
generated free module in degree zero. 
By the transitivity sequence, $L_{A/\mathbb{Z}_p[t_1, \dots, t_n]}$ is a perfect $A$-module. 
Thus, by the Quillen spectral sequence, the homotopy groups of
$\HH(A/\mathbb{Z}[q_1, \dots, q_n])$  and thus 
$\THH(A/\mathbb{Z}[q_1, \dots, q_n])$
are
finitely generated $A$-modules.
Compare also \cite{DundasMorrow} for general finite generation results.

Moreover, after base-change $S^0[q_1,
\dots, q_n] \to S^0$ sending $q_i \mapsto 0$, one obtains B\"okstedt's
calculation $\THH(k)_* \simeq k[\sigma]$. Since the homotopy groups of $
\THH(A/S^0[q_1, \dots, q_n])$ are finitely generated $A$-modules, and $A$ is
local, the result follows. 
\end{proof}

Let $A$ be as above. 
Given a 
smooth and proper $A$-linear stable $\infty$-category $\mathcal{C}$, one 
can consider the invariant $\THH( \mathcal{C}/S^0[q_1, \dots, q_n])$, which
naturally takes values in the symmetric monoidal $\infty$-category
$\md_{\THH(A/S^0[q_1, \dots, q_n]}( \Sp^{BS^1})$. 
This produces a one-parameter deformation of Hochschild homology over $A$, and
it is particularly well-behaved (at least for smooth and proper 
$A$-linear stable $\infty$-categories) 
by the following result. 
\begin{proposition} 
Let $A$ be an $F$-finite regular local ring with system of parameters $t_1,
\dots, t_n$ and perfect residue field $k$.
Any dualizable object in $\md_{\THH(A/S^0[q_1, \dots, q_n])}( \Sp^{BS^1})$ is
perfect. 
\end{proposition} 
\begin{proof} 
This follows by regularity from \cite[Theorem 2.15]{AMN}. \end{proof}

\begin{proposition} 
\label{perfectperf}
Let $A$ be an $F$-finite regular local ring. 
Let $N \in \md_{\THH(A)}( \Sp^{BS^1})$ be dualizable. Then $N \otimes_{\THH(A)}
\THH(A_\mathrm{perf})$ is perfect. 
\end{proposition} 
\begin{proof} 
In fact, we have a factorization
of $\mathbf{E}_\infty$-rings with $S^1$-action
\[ \THH(A) \to \THH( A/S^0[q_1, \dots, q_n]) \to \THH(
A_{\mathrm{perf}}/S^0[q_1^{1/p^\infty}, \dots, q_n^{1/p^\infty}]) \simeq
\THH(A_{\mathrm{perf}}).  \]
Here we use the general observation that for an $\einf$-$S^0[q_1^{1/p^\infty}, \dots,
q_n^{1/p^\infty}]$-algebra $R$, 
the relative Hochschild homology 
$\THH( R/ S^0[q_1^{1/p^\infty}, \dots, q_n^{1/p^\infty}])$ and the absolute
$\THH(R)$ agree after $p$-adic completion; this follows easily from 
$\THH(  S^0[q_1^{1/p^\infty}, \dots, q_n^{1/p^\infty}])/p \simeq
S^0[q_1^{1/p^\infty}, \dots, q_n^{1/p^\infty}]/p$ by perfectness (cf. \cite[Prop. 11.7]{BMS2}). 

We have just seen that $N \otimes_{\THH(A)} \THH( A/S^0[q_1, \dots, q_n])$ is
perfect in 
$\md_{\THH(A/S^0[q_1, \dots, q_n])}( \Sp^{BS^1})$; base-changing up to
$\THH(A_{\mathrm{perf}})$, the result follows. 
\end{proof}

Once more, we make the statement for Hochschild homology of categories, or
more generally for noncommutative motives. 
Let $A$ be an $F$-finite regular noetherian ring with lift $\widetilde{A}$ to
$\mathbb{Z}/p^2$. 
We use, again, the $\infty$-category $\nmot_A$, its subcategory
$\nmot_A^\omega$ generated by the motives of smooth and proper $A$-linear
$\infty$-categories, and the Hochschild homology
functor $\HH(\cdot/A): \nmot_A \to \md_A^{BS^1}$. 

\begin{corollary} 
Let $X \in \nmot_A^\omega$. 
Suppose that $X$ lifts to an object of $\nmot_{\widetilde{A}}^\omega$ and that
$\HH_i(X/A)$ 
vanishes for  
$i \notin [a,b]$ for $b-a \leq 2p-2$. 
Then the Hochschild homology groups $\HH_i(X/A)$ are finitely
generated projective $A$-modules and the $S^1$-action is formal. \end{corollary}

\bibliographystyle{halpha}
\bibliography{kaledin}

\end{document}